\documentclass[reqno]{gthack}
\usepackage[margin=1.19in, left=1.35in, right=1.1in]{geometry}

\usepackage{graphicx, mathtools}
\usepackage{comment}
\usepackage[margin=1cm]{caption}
\usepackage{enumerate,enumitem}
\usepackage{tikz}
\usetikzlibrary{calc, matrix, cd, arrows, decorations.markings}
\usetikzlibrary{decorations.pathmorphing}
\usetikzlibrary{decorations.pathreplacing}

\usepackage[all]{xy}
\xyoption{import,dvips,rotate}

\let\amsamp=&

\theoremstyle{plain}
\newtheorem{proposition}{Proposition}[section]
\newtheorem{theorem}{Theorem}[section]
\newtheorem{corollary}{Corollary}[section]
\newtheorem{lemma}{Lemma}[section]

\theoremstyle{definition}
\newtheorem{definition}{Definition}[section]

\newtheorem{example}{Example}[section]

\theoremstyle{remark}

\let\fullref\autoref

\makeatletter
\let\c@corollary=\c@theorem
\let\c@proposition=\c@theorem
\let\c@lemma=\c@theorem
\let\c@remark=\c@theorem
\let\c@definition=\c@theorem
\let\c@notation=\c@theorem
\let\c@construction=\c@theorem
\let\c@example=\c@theorem
\let\c@equation\c@theorem
\let\c@qyestion\c@theorem
\let\c@figure=\c@figure
\makeatother

\def\makeautorefname#1#2{\expandafter\def\csname#1autorefname\endcsname{#2}}
\makeautorefname{theorem}{Theorem}%
\makeautorefname{corollary}{Corollary}%
\makeautorefname{proposition}{Proposition}%
\makeautorefname{lemma}{Lemma}%
\makeautorefname{remark}{Remark}%
\makeautorefname{definition}{Definition}%
\makeautorefname{notation}{Notation}%
\makeautorefname{section}{Section}%
\makeautorefname{construction}{Construction}%
\makeautorefname{example}{Example}%
\makeautorefname{question}{Question}%
\makeautorefname{figure}{Figure}%

\newcommand{\wt}{\widetilde}

\newcommand{\ol}{\overline}
\newcommand{\sm}{\setminus}
\newcommand{\Cz}{\mathbb{C}_\omega}

\newcommand{\N}{\mathbb{N}}

\newcommand{\im}{\operatorname{Im}}

\newcommand{\Id}{\operatorname{Id}}

\newcommand{\Ext}{\operatorname{Ext}}

\newcommand{\Tor}{\operatorname{Tor}}

\newcommand{\Hom}{\operatorname{Hom}}

\newcommand{\sgn}{\operatorname{sgn}}

\newcommand{\coker}{\operatorname{coker}}

\renewcommand{\epsilon}{\varepsilon}
\renewcommand{\phi}{\varphi}

\begin{document}
\title{A lower bound for the doubly slice genus from signatures}

\begin{abstract}
The doubly slice genus of a knot in the 3-sphere is the minimal genus among unknotted orientable surfaces in the 4-sphere for which the knot arises as a cross-section. We use the classical signature function of the knot to give a new lower bound for the doubly slice genus. We combine this with an upper bound due to C.~McDonald to prove that for every nonnegative integer~$N$ there is a knot where the difference between the slice and doubly slice genus is exactly~$N$, refining a result of W.~Chen which says this difference can be arbitrarily large.
\end{abstract}

\author[Patrick Orson]{Patrick Orson}
\address{Department of Mathematics, Boston College, USA}
\email{patrick.orson@bc.edu}

\author[Mark Powell]{Mark Powell}
\address{Department of Mathematical Sciences, Durham University, UK}
\email{mark.a.powell@durham.ac.uk}

\begin{center}
\maketitle
  \end{center}

\section{Introduction}

In what follows all manifolds are topological, compact, and oriented, and embeddings are locally flat, although our results also hold in the smooth category.
A basic $4$-dimensional measurement for the complexity of a knot $K\subset S^3$ is the \emph{slice genus} $g_4(K)$, defined as the minimal genus among connected properly embedded surfaces in $D^4$ that have the knot as boundary. Doubling such a surface along its boundary produces a closed connected surface in $S^4$ for which the knot appears as a cross section. This doubled surface will be genus minimising among surfaces in $S^4$ for which the knot appears as a cross section, but will in general be a knotted surface embedding.

A connected surface in $S^4$ is \emph{unknotted} if it bounds an embedded $3$-dimensional handlebody in~$S^4$. Unknotted surfaces with the knot~$K$ as cross section are easily produced by doubling a Seifert surface for~$K$ that has been pushed in to $D^4$. The \emph{doubly slice genus} $g_{ds}(K)$, first defined in~\cite[\textsection 5]{MR3425633}, is the minimal genus among unknotted surfaces in $S^4$ for which the knot arises as a cross-section. Writing $g_3(K)$ for the minimal genus among Seifert surfaces for~$K$, it is immediate from the above discussion that
\[2g_4(K) \leq g_{ds}(K) \leq 2g_3(K).\]
Further comparison of these quantities is fairly subtle, but we will show in this article that classical abelian knot invariants can be employed for this purpose.

A choice of Seifert surface for a knot $K\subset S^3$ and a choice of basis for the first homology gives rise to a Seifert matrix~$V$. Then given $\omega\in S^1 \subset \C$ the \emph{$\omega$-signature} of~$K$ is defined as the signature of the complex hermitian matrix
\[
\sigma_\omega(K):= \sgn\left((1-\omega)V+(1-\omega^{-1})V^T\right).
\]

\begin{theorem}\label{thm:main}
  Let $K$ be a knot in $S^3$. The doubly slice genus of $K$ is at least
  \[
  g_{ds}(K) \geq \max_{\omega \in S^1 \sm \{1\}}|\sigma_\omega(K)|.
  \]
  \end{theorem}

Let $\Delta_K(t)$ denote the Alexander polynomial of $K$. A classical lower bound for the slice genus is that for every $\omega \in S^1$ such that $\Delta_K(\omega) \neq 0$, we have $|\sigma_{\omega}(K)| \leq 2 g_4(K)$ \cite{MR388373}. It follows that $|\sigma_{\omega}(K)| \leq g_{ds}(K)$ for these~$\omega$.   Our theorem refines this, since it also applies when $\omega$ is a root of the Alexander polynomial of $K$.  Given a slice knot $K$, in other words a knot with $g_4(K)=0$, and for $\omega \in S^1$ such that $\Delta_K(\omega) \neq 0$, we have $\sigma_{\omega}(K)=0$.  Therefore the classical bound contains no information on the doubly slice genus for slice knots.

On the other hand, for every $\nu\in S^1\sm\{1\}$ that is the root of some Alexander polynomial there exists a slice knot $K$ for which $\sigma_\omega(K)$ is nontrivial exactly at $\omega=\nu, \overline{\nu}$ \cite[Corollary 2.1]{Cha-Livingston:2002-1}. For any $N\in\N$, \fullref{thm:main} applied to the $N$-fold connected sum of such a knot with itself immediately produces a slice knot with doubly slice genus at least $N$, recovering a theorem of Chen~\cite{ChenArxiv}, which we discuss below. In the following result we obtain a refinement of such examples.

\begin{theorem}\label{thm:examples}
For each $N \in \mathbb{N}$ there exists a slice knot $K_N$ with $g_{ds}(K_N)=N$. In fact, we may take $K_N=\#^NJ$, the $N$-fold connected sum of $J$ with itself for some
\begin{align*}
  J \in \{& 8_{20},  10_{87}, 10_{140}, 11a28, 11a58, 11a165, 12a189, 12a377, 12a979,  12n56, \\ & 12n57,  12n62, 12n66, 12n87, 12n106, 12n288, 12n501, 12n504, 12n582, 12n670, 12n721\}.
  \end{align*}
  Here we use the notation of KnotInfo~\cite{knotinfo}.
\end{theorem}

\begin{proof} The 21 knots listed are slice knots, found by searching the KnotInfo tables, of at most 12 crossings, whose $\omega$-signature equals $1$ for some $\omega \in S^1$ with $\Delta_J(\omega)=0$. As the lower bound of \fullref{thm:main} is additive under connected sum we therefore have $g_{ds}(K_N)\geq N$.

We will show in \fullref{prop:bandmoves} that each of these knots admits a slice disc on which the radial Morse function has two minima and one saddle point i.e.~$J$ arises from one band move on the 2-component unlink. The following theorem of Clayton McDonald therefore shows that each of the knots $J$ has doubly slice genus at most $1$, and that $K_N$ therefore has $g_{ds}(K_N)\leq N$.
\end{proof}

\begin{theorem}[McDonald~{\cite[Theorem~3.2]{McDonald}}]\label{thm:mcdonald}
Let $K \subset S^3$ be a knot and let $\Sigma$ be a smoothly embedded surface in $D^4$ such that the radial Morse function restricts to a Morse function on $\Sigma$ with~$b$ saddle points and no maxima. Then $g_{ds}(K)\leq b$.
\end{theorem}

\begin{corollary}[to \fullref{thm:examples}]
  Let $M,N$ be nonnegative integers with $M$ even and $M \leq N$. There exists a knot $K$ with $M=2g_4(K)$ and $N= g_{ds}(K)$.
\end{corollary}

\begin{proof}
  Let $J$ be the mirror image of the knot $5_2$. This has $g_4(J)= g_3(J)= g_{ds}(J) =1$, and $\sigma_{\omega}(J) =2$ for $\omega := e^{\pi i/3}$, which is not a root of the Alexander polynomial.  The knot $L:= 8_{20}$ has $g_4(L)=0$, but $\sigma_{\omega}(L) = 1$ and $g_{ds}(L) =1$. Taking
  \[K:= \big(\#^{M/2} J\big) \# \big(\#^{N-M} L\big)\]
  yields a knot with $2g_4(K) \leq M$ and $g_{ds}(K) \leq N$.
  Then $\sigma_{\omega}(K) = N$, so $g_{ds}(K) =N$ by \fullref{thm:main}.
Since $|\sigma_{\rho}(K)| \leq 2 g_4(K)$ except for finitely many values of $\rho \in S^1$, the averaged signature function defined  by \[\ol{\sigma}_{e^{i\pi\theta}}(K) := \frac{1}{2} \left(\lim_{\varphi \to \theta^+} \sigma_{e^{i\pi\varphi}}(K) + \lim_{\varphi \to \theta^-} \sigma_{e^{i\pi\varphi}}(K)\right)\] satisfies $|\ol{\sigma}_{\rho}(K)| \leq 2 g_4(K)$ for all $\rho \in S^1$.    Then $\ol{\sigma}_{\omega}(K) = M$ so $2g_4(K) = M$.
\end{proof}

\subsubsection*{Connections to previous work}

A knot $K$ is \emph{doubly slice} if $g_{ds}(K)=0$, and the doubly slice genus is a measure of how far a knot is from being doubly slice. The first detailed study of doubly slice knots, and the related algebra, was made by Sumners~\cite{MR290351}. Further foundational algebraic studies, related to the work in this article, are those of Stoltzfus~\cite{MR521738} and Levine~\cite{MR1004605}.

Instead of the doubly slice genus, a different measure of the failure of a knot to be doubly slice was studied by Cherry Kearton \cite{MR616563}. Given a slice knot $K$, he considered the minimal complex dimension of $H_1(S^4\sm J;\C[t,t^{-1}])$ among all knotted 2-spheres $J\subset S^4$ with cross-section $K$. He gave lower bounds for his invariant arising from signature obstructions. The signatures he considered are the $(p,i)$-signatures of the Blanchfield form (see \cite{MR1004605}), and it is known that these signatures can be used to compute the $\omega$-signatures of $K$ \cite[Theorem 2.3]{MR1004605}, tempting one to imagine a connection to the results of this paper. But despite the similar flavour of the invariants he uses, Kearton's complexity measure appears to be independent of the doubly slice genus, so there is no clear dependency between his work and ours.

This article was partly inspired by work of Wenzhao Chen~\cite{ChenArxiv}, who ingeniously applied Casson-Gordon invariants to show that for every $N\in\N$, there is a slice knot $K$ with $g_{ds}(K) \geq N$. In particular he proved that $g_{ds}(K) - 2g_4(K)$ can be arbitrarily large.  Casson-Gordon invariants rely on the existence of interesting metabelian representations of the knot group $\pi_1(S^3\sm K)$ and are thus less basic than the $\omega$-signatures in this paper, which can be thought of as arising from the abelianisation of the knot group $\pi_1(S^3\sm K) \to \Z$. While our method refines Chen's theorem, with a more elementary invariant, we cannot recover Chen's examples. These examples, as the original Casson-Gordon examples, are constructed using the Stevedore's knot. With rational coefficients the Stevedore's knot shares a Seifert matrix with $9_{46}$, which is doubly slice. This means Chen's examples have hyperbolic Seifert matrices over the rational numbers, and so for all $\omega\in S^1\sm \{1\}$ the $\omega$-signature of his knots vanish.

\subsection*{Outline}

The paper is organised as follows. In \fullref{sec:signature-defects} we recall the signature defect invariants of a 3-manifold with a map to $B\Z$, associated with a cobounding $4$-manifold. We equate the signature defect invariant with $\omega$-signatures.  In \fullref{section:lower-bounds} we use this to prove \fullref{thm:main}. In \fullref{sec:examples} we establish the upper bounds for the examples listed in \fullref{thm:examples}.

\subsection*{Acknowledgements}
We thank Lucia Karageorghis of Durham University, who was supported by an LMS summer undergraduate research fellowship, for help finding the band moves for the knots in \fullref{prop:bandmoves}, as part of her study of the doubly slice genus of the prime knots up to 12 crossings. She was aided by the Kirby calculator/KLO; we are grateful to Frank Swenton for creating this excellent tool.  We are also grateful for the existence of KnotInfo; we thank Chuck Livingston for help interpreting it and for his insightful comments during the preparation of this article.

\section{Signature defects}\label{sec:signature-defects}

Let $R$ be either the ring $\C$ with the involution given by complex conjugation, or the ring of finite complex Laurent polynomials $\C[\Z]\cong\C[t,t^{-1}]$ with involution given by $\sum a_k t^k\mapsto \sum \overline{a_k}t^{-k}$.  An \emph{$R$-module} will mean a left $R$-module unless otherwise stated, and $\overline{\phantom{P}}$ will denote the use of the involution to switch a left $R$-module to a right $R$-module or vice-versa.

A CW pair of connected topological spaces $(X,Y)$ is \emph{over $\Z$} if $X$ is equipped with a homomorphism $\phi\colon \pi_1(X)\to \Z$. We write $(X,Y,\phi)$ for these data, or $(X,\phi)$ if $Y=\emptyset$. Write $p\colon\widetilde{X}\to X$ for the cover corresponding to $\phi$ and $\widetilde{Y}=p^{-1}(Y)$ for the corresponding cover of~$Y$. Given a map of rings with involution $\alpha\colon\C[\Z]\to R$, the ring $R$ becomes an $(R,\C[\Z])$-bimodule, and there are associated twisted homology and cohomology modules over~$R$
\begin{align*}
H_r(X,Y;\alpha)&:=H_r(R\otimes_{\alpha} C_*(\wt{X},\widetilde{Y};\C)),\\
H^r(X,Y;\alpha)&:=H_r(\Hom_{\C[\Z]}(\overline{C_*(\wt{X},\widetilde{Y};\C)},R)).
\end{align*}
Note we are abusing notation in suppressing the particular $\phi$ being used, but for all applications in this article the choice of  $\phi$ will be understood, so this should cause no confusion.

Setting $\alpha$ to be the identity map $\Id\colon \C[\Z]\to \C[\Z]$ returns the ordinary complex coefficient homology and complex coefficient cohomology with compact support of the cover $(\wt{X},\wt{Y})$. We denote these by $H_r(X,Y;\C[\Z])$ and $H^r(X,Y;\C[\Z])$ respectively.

For each $\omega\in S^1\sm\{1\}$ there is a map of rings with involution
\[
\alpha_\omega\colon \C[t,t^{-1}]\to\C;\qquad \alpha_\omega(t)=\omega.
\]
The map $\alpha_\omega$ induces a $(\C,\C[\Z])$-bimodule structure on $\C$ and we will write $\C_\omega$ when we wish to emphasise this structure is being used. We will write
\[
H_r(X,Y;\C_\omega):=H_r(X,Y;\alpha_\omega),\qquad H^r(X,Y;\C_\omega):=H_r(X,Y;\alpha_\omega).
\]

Now consider $(X,\phi)$ where $X$ is a compact, oriented $n$-dimensional manifold with (possibly empty) boundary. Let $PD\colon H^{n-k}(X;\C_\omega)\to H_{k}(X,\partial X;\C_\omega)$ denote the Poincar\'{e} duality isomorphism. Define a map of complex vector spaces
\[
\lambda_\omega(X)\colon H_k(X;\C_\omega)\to H_k(X,\partial X;\C_\omega)\xrightarrow{PD^{-1}}H^{n-k}(X;\C_\omega)\xrightarrow{\operatorname{ev}} \overline{\Hom_\C(H_{n-k}(X;\C_\omega),\C)},
\]
where $\operatorname{ev}$ denotes the \emph{evaluation} map given by $\operatorname{ev}([f])([z\otimes x])=z\cdot\overline{f(x)}$. The map $\lambda_\omega(X)$ determines a pairing
\[
H_{n-k}(X;\Cz)\times H_{k}(X;\Cz)\to \C;\qquad (x,y)\mapsto \lambda_\omega(X)(y)(x),
\]
which is hermitian and sesquilinear but in general is degenerate. In particular, when $n=2k$, we may take the signature of this complex hermitian pairing, denoted $\sigma(\lambda_\omega(X))\in\Z$.

\begin{definition}For $W$ a compact, oriented $4$-manifold with (possibly empty) boundary, over $\Z$, the (middle dimensional) $\C_\omega$-coefficient \emph{intersection form} is the hermitian sesquilinear form $(H_2(W;\C_\omega),\lambda_{\omega}(W))$.
\end{definition}

\begin{definition}
Let $(M,\phi)$ be a closed, connected, oriented $3$-manifold over $\Z$. A \emph{null-bordism} of $(M,\phi)$ is a pair $(W,\psi)$ consisting of a compact, connected, oriented $4$-manifold~$W$ with boundary $\partial W= M$ and a homomorphism $\psi\colon \pi_1(W)\to \Z$ such that $\psi|_{\partial W}=\phi$.
\end{definition}

Given a null-bordism $(W,\psi)$ of $(M,\phi)$, we define the \emph{$\omega$-signature defect}
\[
\sigma_\omega(M): = \sigma(\lambda_{\omega}(W))-\sigma(W).
\]
(We are abusing notation in suppressing the particular $\phi$ and $\psi$.)

\begin{proposition}\label{thm:welldefined}
Given a closed, connected, oriented $3$-manifold $(M,\phi)$ over $\Z$, and any $\omega\in S^1\sm\{1\}$, the $\omega$-signature defect $\sigma_\omega(M)$ is defined and well-defined, independent of the choice $(W,\psi)$.
\end{proposition}

\begin{proof} Because $\Omega_3(B\Z)=0$, there always exists a null-bordism $(W,\psi)$ for $(M,\phi)$. The proof that the resultant $\omega$-signature defect is independent of the choice of $(W,\psi)$ is a well-known Novikov additivity argument, as we now outline. First, write $i\colon H_2(M;\Cz)\to H_2(W;\Cz)$ for the inclusion induced map. The image of~$i$ lies in the kernel of~$\lambda_\omega(X)$ by exactness of the long exact sequence of the pair $(W,M)$. The restriction of $\lambda_\omega(X)$ to the quotient $H_2(W;\Cz)/i(H_2(M;\Cz))$ determines a nonsingular pairing \cite[Proposition 5.3 (i)]{MR3609203}. Thus the signature of $\lambda_\omega(W)$ and the signature of its restriction to $H_2(W;\Cz)/i(H_2(M;\Cz))$ agree. We now refer the reader to the proof of \cite[Proposition 5.3 (ii)]{MR3609203} for the completion of the argument.
\end{proof}

\begin{example}
The main example we are interested in is the closed, oriented $3$-manifold $M_K$ obtained by 0-framed Dehn surgery on $S^3$ along an oriented knot $K$. The orientation on the knot determines a natural map $\phi_K\colon \pi_1(M_K)\to \Z$ via abelianisation.

The associated $\C[\Z]$-coefficient homology $H_*(M_K;\C[\Z])$ is torsion; that is there exists a Laurent polynomial $p\in \C[\Z]$ such that $p\cdot H_*(M_K;\C[\Z])=0$.
\end{example}

\begin{example}\label{ex:surface}Let $G\subset D^4$ be a properly embedded, connected genus $g$ surface with one boundary component, homeomorphic to $\overline{\Sigma_g \sm D^2} =:\Sigma_{g,1}$.  Let $\nu G$ be an open tubular neighbourhood extending an open tubular neighbourhood of the boundary knot $K\subset S^3$. Let $H_g$ denote the $3$-dimensional handlebody of genus $g$ and let $\Sigma_g$ be its boundary. By choosing a disc $D^2\subset \partial H_g$, decompose the boundary of $H_g\times S^1$ as
\[
\partial(H_g\times S^1)=( \Sigma_{g,1} \times S^1) \,\cup_{S^1\times S^1} \,(D^2\times S^1).
\]
Glue the exterior of $G$ to $H_g\times S^1$, along $\Sigma_{g,1} \times S^1$ to form
\[
W:=(D^4\sm \nu G)\,\cup_{G\times S^1}\,(H_g\times S^1),
\]
a compact, connected, oriented $4$-manifold with boundary $M_K$, the $0$-surgery on~$K$.  Mayer-Vietoris calculations give
\[
H_k(W;\Z)\cong\left\{\begin{array}{lll}\Z&k=0,&\\
\Z&k=1, &\text{generated by a meridian of $G$,}\\
\Z^{2g}&k=2, &\\
0&\text{otherwise.}&\end{array}\right.
\]
In particular, the abelianisation of $\phi\colon\pi_1(M_K)\to \Z$ extends to $\psi\colon \pi_1(W)\to \Z$ so that $(W,\psi)$ is a null-bordism of $(M_K,\phi)$.  Note that the homology is independent of the choice of identification of $G$ with $\Sigma_{g,1} \subset \partial H_g$.
\end{example}

\begin{lemma}\label{ex:seifertsurface} Let $K\subset S^3$ be an oriented knot and let $M_K$ be the 0-surgery manifold. For any $\omega\in S^1\sm\{1\}$ and there is equality
\[
\sigma_\omega(M_K)=\sigma_\omega(K).
\]
\end{lemma}

\begin{proof}
As the $\omega$-signature is well-defined, independent of choice of null-bordism,  it suffices to find a single null-bordism of $M_K$ over $\Z$ such that the signature of the $\C_\omega$-coefficient intersection form agrees with $\sigma_\omega(K)$. Perform the construction of \fullref{ex:surface} on a pushed in Seifert surface $F$ for $K$. In this case it is shown in shown by Ko \cite[pp.~538-9]{MR1010376} (see also Cochran-Orr-Teichner~\cite[Lemma 5.4]{MR2031301}) that in some basis the resultant $\C_\omega$-coefficient intersection form of $W_F$  has matrix $(1-\omega)V+(1-\omega^{-1})V^T$, where $V$ is a Seifert matrix associated to $F$.  Moreover the ordinary signature $\sigma(W_F)=0$, so the defect satisfies
\[\sigma_{\omega}(M_K) = \sigma(\lambda_{\omega} (W_F)) - \sigma(W_F) = \sigma_{\omega}(K).\hfill\qedhere\]
\end{proof}

\section{A lower bound on \texorpdfstring{$g_{ds}$}{the doubly slice genus}}\label{section:lower-bounds}

Let $K \subset S^3$ be an oriented knot, let $G_1, G_2 \subset D^4$ be locally flat, connected, compact, orientable, embedded surfaces with boundary $K$, such that $S= {G}_1 \cup_K {G}_2$ is an unknotted surface in $S^4$ of genus $g$.

Perform the construction described in \fullref{ex:surface} on each of $G_1$ and $G_2$ to obtain $W_1$ and $W_2$ respectively. Define
\[
V:=W_1\cup_{M_K}-W_2.
\]

Observe that $V=(S^4\sm\nu S)\cup_{\Sigma_g\times S^1} (H_g\times S^1)$, where $H_g$ denotes the 3-dimensional handlebody of genus $g$ and $\Sigma_g=\partial H_g$.

A straightforward Seifert-Van Kampen argument shows that $\pi_1(V) \cong \Z$. Various Mayer-Vietoris calculations give
\[
H_k(V;\Z)\cong\left\{\begin{array}{lll}\Z&k=0,&\\
\Z&k=1, &\text{generated by a meridian of $\Sigma_g$,}\\
\Z^{2g}&k=2, &\\
0&\text{otherwise.}&\end{array}\right.
\]

We now derive a series of technical lemmas we will use in the proof of \fullref{thm:main}

\begin{lemma}\label{lem:squarepresent}
Let $T$ be a finitely generated, torsion ${\C[\Z]}$-module, and let $\omega\in S^1\sm\{1\}$. Then $\dim_\C\Tor^{\C[\Z]}_1(T,\C_\omega)=\dim_\C (\C_\omega\otimes_{\C[\Z]} T)$.
\end{lemma}

\begin{proof}
By the structure theorem for finitely generated modules over a principal ideal domain, there exists an injective map $A\colon P_1\hookrightarrow P_0$ such that $T\cong P_0/A(P_1)$ and so that $P_1, P_0$ are free ${\C[\Z]}$-modules of the same rank. The functor $\C_\omega\otimes_{\C[\Z]}-$ induces an exact sequence
\[
\Tor_1^{\C[\Z]}(P_0,\C_\omega)\to \Tor_1^{\C[\Z]}(T,\C_\omega) \to \C_\omega\otimes_{\C[\Z]} P_1\xrightarrow{\Id \otimes A} \C_\omega\otimes_{\C[\Z]} P_0\to \C_\omega\otimes_{\C[\Z]} T\to 0.
\]
The leftmost term is 0 because $P_0$ is free. As $P_1$ and $P_0$ have the same free rank, $\C_\omega\otimes_{\C[\Z]} P_1$ and $\C_\omega\otimes_{\C[\Z]} P_0$ have the same complex dimension. The sequence has vanishing Euler characteristic because it is exact, so the claimed result follows.
\end{proof}

\begin{lemma}\label{lem:littlefact} For a space $X$ over $\Z$ with $H_0(X;{\C[\Z]})\cong\C$, and for $\omega\in S^1\sm\{1\}$ we have
\[
 H_0(X;\C_\omega)=0 \quad\text{and}\quad H_1(X;\C_\omega)\cong \C_\omega\otimes_{\C[\Z]}H_1(X;\C[\Z]).
\]
\end{lemma}

\begin{proof}
First, $\C\cong\C[t,t^{-1}]/(t-1)$ as a $\C[\Z]$-module, so that $\C_\omega\otimes_{\C[\Z]}\C=0$ since $\omega\neq 1$. This immediately gives $\C_\omega\otimes_{\C[\Z]} H_0(X;\C[\Z])=0$ and by \fullref{lem:squarepresent} we also have  $\Tor^{\C[\Z]}_1(H_0(X;{\C[\Z]}),\C_\omega)=0$. The result now follows from the Universal Coefficient Theorem.
\end{proof}

\begin{lemma}\label{lem:MVcrop}
With $V=W_1\cup_{M_K}-W_2$ as described above and $\omega\in S^1\setminus\{1\}$,
\[
H_1(V;\Cz)=0,\qquad H_3(V;\Cz)=0,\qquad \text{ and }\qquad \dim_{\C} H_2(V;\Cz) = 2g,
\]
so that the Mayer-Vietoris sequence for $V$ with $\mathbb{C}_\omega$ coefficients becomes
\[
0 \to H_2(M_K) \xrightarrow{} H_2(W_1) \oplus H_2(W_2) \to
\C^{2g}
\to H_1(M_K) \xrightarrow{} H_1(W_1) \oplus H_1(W_2) \to 0.
\]
\end{lemma}

\begin{proof}
Consider that $\C_\omega\otimes_{\C[\Z]} H_1(V;{\C[\Z]})=0$ since $\pi_1(V) \cong \Z$ implies $H_1(V;{\C[\Z]})=0$.  Since $H_0(V;{\C[\Z]})\cong \C$ and $\omega \neq 1$, this combines with \fullref{lem:littlefact} to give $H_1(V;\Cz)=0$.

Next, we have $H_3(V;\Cz)\cong H^1(V;\Cz)$ by Poincar\'{e} duality. By the Universal Coefficient Theorem for cohomology, $H^1(V;\Cz)\cong \Ext^1_{\C[\Z]}(H_0(V;\C[\Z]),\C_\omega)$. The projective $\C[\Z]$-module resolution \[0\to \C[\Z]\xrightarrow{f} \C[\Z]\to H_0(V;\C[\Z])\to 0,\] where $f\colon p(t)\mapsto (t-1)p(t)$, can be used to compute
\[
\Ext^1_{\C[\Z]}(H_0(V;\C[\Z]),\C_\omega)=\coker(\Hom_{\C[\Z]}(\C[\Z],\C_\omega)\xrightarrow{f^*} \Hom_{\C[\Z]}(\C[\Z],\C_\omega)).
\]
But $f^*(\phi)=(\omega-1)\phi$, and $\omega\neq 1$, so this module vanishes as required.

Using the integral homology of $V$, we compute the Euler characteristic $\chi(V) = 2g$.  We shall compute it again with $\Cz$-coefficients in order to find the dimension of $H_2(V;\Cz)$.  By \fullref{lem:littlefact} we have $H_0(V;\Cz) =0$, so also $H_4(V;\Cz)=0$ by Poincar\'{e} duality and the Universal Coefficient Theorem. Therefore $H_i(V;\Cz)=0$ for $i \neq 2$, and we have
\[2g = \chi(V) = \chi^{\Cz}(V) = \dim_{\C} H_2(V;\Cz).\hfill\qedhere\]
\end{proof}

\begin{lemma}\label{lem:H3}
For $i=1,2$ there is equality
\[
\dim_\C \im(H_2(M_K;\C_\omega)\to H_2(W_i;\C_\omega)) = \dim_\C H_2(M_K;\C_\omega) - \dim_\C H_1(W_i;\Cz).
\]
\end{lemma}

\begin{proof}
The map $H_1(M_K;\C_\omega)\to H_1(W_i;\C_\omega)$ is surjective by \fullref{lem:MVcrop}. This implies $H_1(W_i,M_K;\Cz) =0$, since $H_0(M_K;\Cz) =0$ by \fullref{lem:littlefact}. Therefore \[H_3(W_i;\Cz) \cong H^1(W_i,M_K;\Cz) \cong H_1(W_i,M_K;\Cz) =0\] by Poincar\'{e} duality and the Universal Coefficient Theorem. For the same reasons, we have
\[
H_3(W_i,M_K;\Cz) \cong H^1(W_i;\Cz) \cong H_1(W_i;\Cz).
\]
Since $H_3(W_i;\Cz)=0$, the long exact sequence of the pair $(W_i,M_K)$ takes the form
\[
0 \to H_3(W_i,M_K;\Cz) \to H_2(M_K;\Cz) \to H_2(W_i;\Cz) \to \cdots.
\]
We deduce that
\begin{align*}
\dim_\C \im(H_2(M_K;\C_\omega)\to H_2(W_i;\C_\omega)) &= \dim_\C H_2(M_K;\C_\omega) - \dim_\C H_3(W_i,M_K;\Cz)\\
&= \dim_\C H_2(M_K;\C_\omega) - \dim_\C H_1(W_i;\Cz).
\end{align*}
as desired.
\end{proof}

We are now in a position to prove \fullref{thm:main}.

\begin{proof}[Proof of \fullref{thm:main}]
Fix $\omega\in S^1\setminus\{1\}$ and let $W_1$, $W_2$ be as above.
Define for $i=1, 2$,
\begin{align*}
\beta&:= \dim_\C H_2(M_K;\C_\omega),\\
n_i&:= \dim_\C   (\C_\omega\otimes_{\C[\Z]} H_2(W_i;\C[\Z])),\\
m_i&:= \dim_\C  \Tor^{\C[\Z]}_1(H_1(W_i;{\C[\Z]}),\C_\omega).
\end{align*}
By the Universal Coefficient Theorem
\[
n_i+m_i=\dim_\C H_2(W_i;\C_\omega)
\quad\text{and}\quad
\beta=\dim_\C \Tor^{\C[\Z]}_1(H_1(M_K;{\C[\Z]}),\C_\omega),
\]
where the latter equality also uses the fact that $H_2(M_K;{\C[\Z]})=0$.

The module $H_1(M_K;{\C[\Z]})$ is ${\C[\Z]}$-torsion. As $H_1(V;\C[\Z])=0$, the map $H_1(M_K;{\C[\Z]})\to H_1(W_1;{\C[\Z]})\oplus H_1(W_2;{\C[\Z]})$ in the Mayer-Vietoris sequence is surjective. Hence $H_1(W_i;{\C[\Z]})$ is torsion for $i=1,2$. By \fullref{lem:squarepresent} we deduce
\begin{align*}
\beta&= \dim_\C (\C_\omega\otimes_{\C[\Z]} H_1(M_K;{\C[\Z]})),\\
m_i&= \dim_\C ( \C_\omega\otimes_{\C[\Z]} H_1(W_i;{\C[\Z]})).
\end{align*}
For each of the spaces $X= M_K, W_1, W_2$, \fullref{lem:littlefact} implies $\C_\omega\otimes_{\C[\Z]} H_1(X;{\C[\Z]})\cong H_1(X;\C_\omega)$ so that we furthermore obtain
\begin{align*}
\beta&= \dim_\C H_1(M_K;\C_\omega),\\
m_i&= \dim_\C H_1(W_i;\C_\omega).
\end{align*}

By \fullref{ex:seifertsurface} we have $|\sigma_\omega(K)| \leq \dim_\C H_2(W_i;\C_\omega)$. However, recall that the image of $H_2(M_K;\C_\omega)\to H_2(W_i;\C_\omega)$ lies in the kernel of $\lambda_\omega(W_i)$, so that moreover
\begin{align*}
|\sigma_\omega(K)| &\leq \dim_{\C} H_2(W_i;\Cz) - \dim_\C \im(H_2(M_K;\C_\omega)\to H_2(W_i;\C_\omega)) \\
& =  \dim_{\C} H_2(W_i;\Cz) - \big(\dim_\C H_2(M_K;\C_\omega) - \dim_\C H_1(W_i;\Cz)\big) \\
 &=  (n_i+m_i)-(\beta - m_i) \\ &= n_i + 2m_i - \beta,
\end{align*}
where in the second line we have used \fullref{lem:H3}.
Taking the sum for $i=1,2$ we obtain:
\begin{align}\label{inequality-1}
2|\sigma_\omega(K)| \leq   n_1 + n_2 + 2m_1 + 2m_2 - 2\beta.\tag{$\ast$}
\end{align}

We saw in \fullref{lem:MVcrop} that $H_1(M_K;\Cz) \to H_1(W_1;\Cz) \oplus H_1(W_2;\Cz)$ is surjective, so that
\[ m_1 + m_2  \leq  \beta.\]
It follows that $2m_1 +2m_2 - 2\beta \leq 0$, so combining this with \eqref{inequality-1} we have
\begin{equation}\label{inequality-2}
 2|\sigma_\omega(K)| \leq n_1 + n_2 + 2m_1 +2m_2 - 2\beta \leq n_1 + n_2.\tag{$\dagger$}
\end{equation}

Finally, we calculate the Euler characteristic for the section of the Mayer-Vietoris sequence of $V=W_1\cup_{M_K}-W_2$ obtained in \fullref{lem:MVcrop} as
\[
0=\beta-(n_1+m_1+n_2+m_2)+2g-\beta+(m_1+m_2),
\]
so that $2g=n_1+n_2$.
Substituting into \eqref{inequality-2} yields $2|\sigma_\omega(K)| \leq 2g$ and hence $|\sigma_\omega(K)| \leq g$.

Since this is true for all $\omega\in S^1\sm\{1\}$ and all pairs of slice surfaces that glue to be unknotted, the claimed result follows.
\end{proof}

\section{Examples of band moves}\label{sec:examples}

A \emph{ribbon surface} for a knot $K \subset S^3$ is a smoothly embedded surface $\Sigma \subset D^4$ with $\partial \Sigma =K$, such that the radial function $D^4 \to [0,1]$ restricts to a Morse function on $\Sigma$ whose critical points are of index either $0$ or $1$.

\begin{definition}
  The \emph{ribbon surface band number} $b(K)$ of a knot $K$ is the minimal number of index 1 critical points, among all ribbon surfaces $\Sigma$ for $K$.
\end{definition}

The following proposition, combined with \fullref{thm:mcdonald} of McDonald, gives the promised upper bounds on $g_{ds}$ that complete the proof of \fullref{thm:examples}.

\begin{proposition}\label{prop:bandmoves} The ribbon surface band number $b(J)=1$ for each of the knots
\begin{align*}
  J \in \{& 8_{20},  10_{87}, 10_{140}, 11a28, 11a58, 11a165, 12a189, 12a377, 12a979,  12n56, \\ & 12n57,  12n62, 12n66, 12n87, 12n106, 12n288, 12n501, 12n504, 12n582, 12n670, 12n721\}.
  \end{align*}
\end{proposition}

\begin{proof} It suffices to exhibit a single band move on $J$ that produces a 2-component unlink. The required band moves are shown in the diagrams of \fullref{fig:1} and \fullref{fig:2}.
  \begin{figure}
  \begin{center}
    \begin{tikzpicture}
    \node[inner sep=0pt] at (0,0)
    {\includegraphics[width=.22\textwidth]{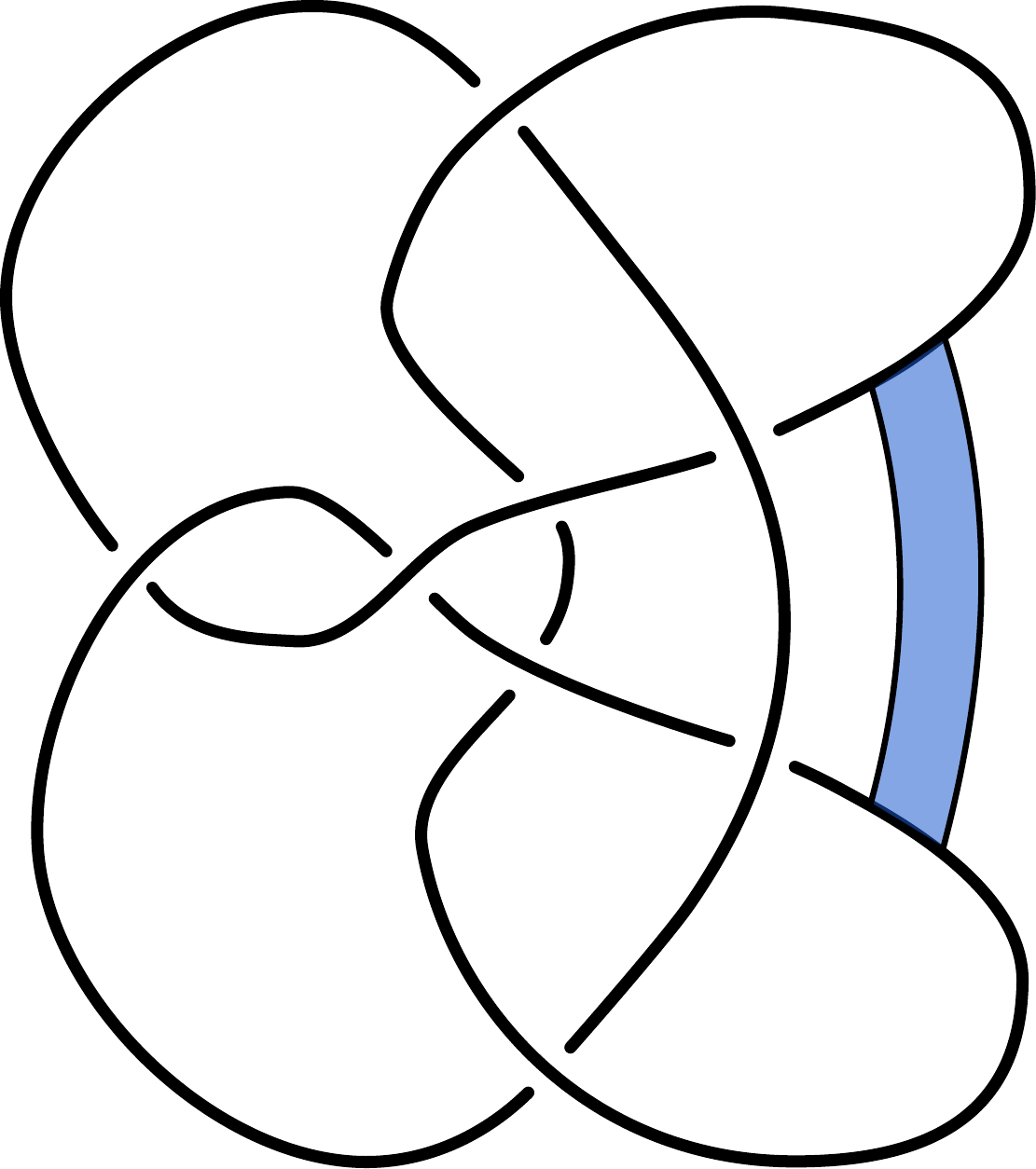}};
    \node[inner sep=0pt] at (5.5,0)
    {\includegraphics[width=.24\textwidth]{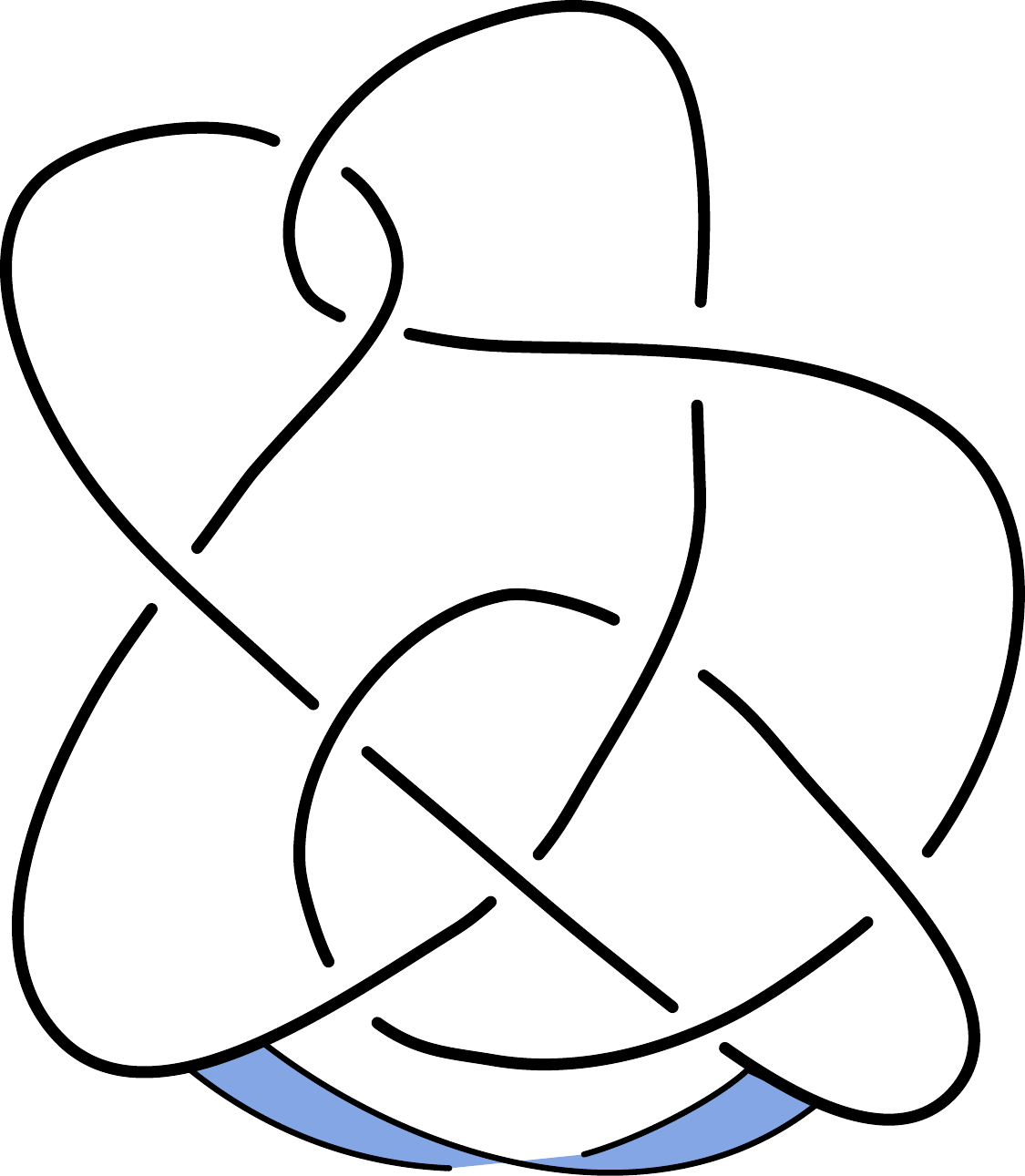}};
    \node[inner sep=0pt] at (11,0)
    {\includegraphics[width=.26\textwidth]{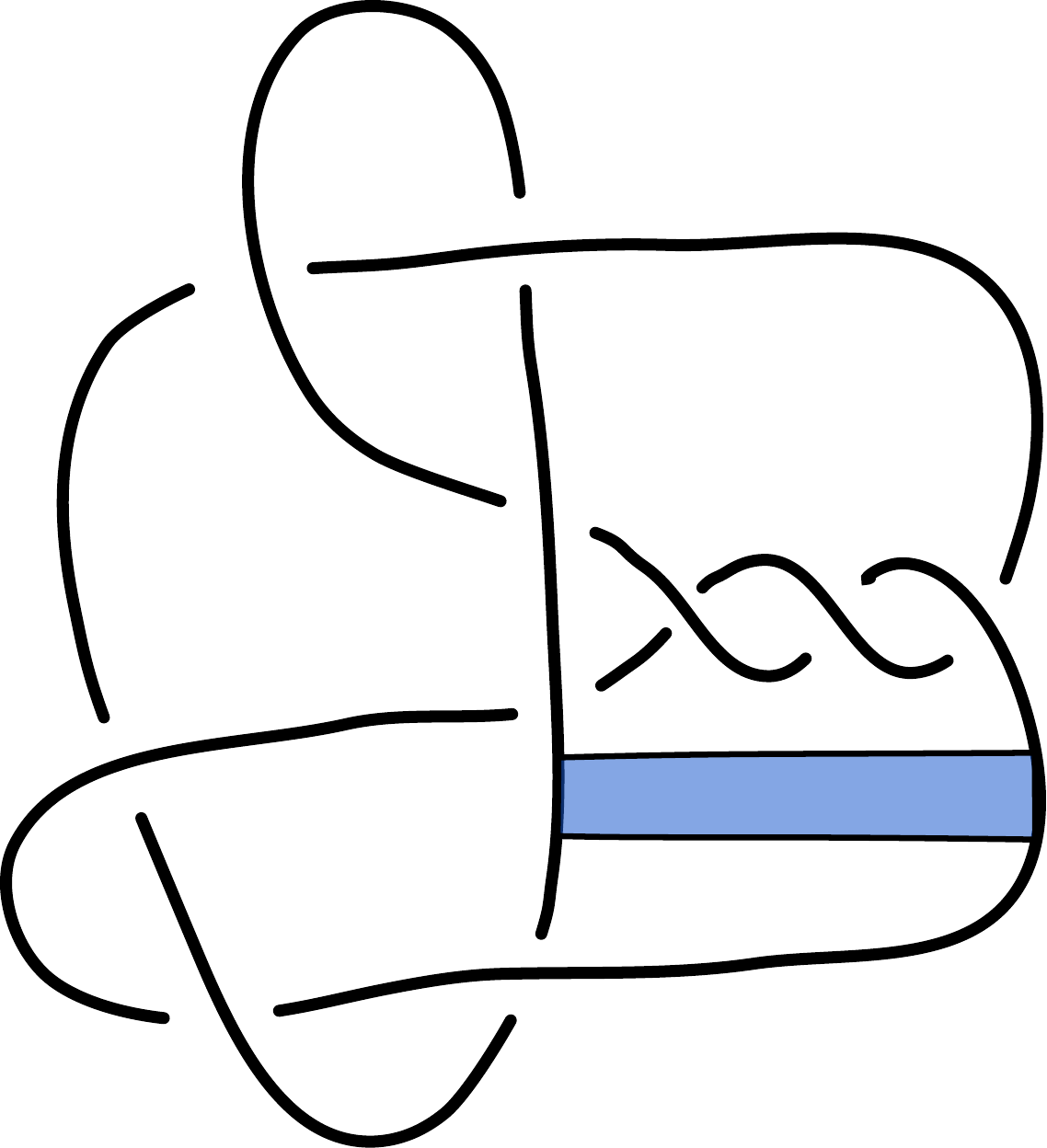}};

        \node[inner sep=0pt] at (0,-2.5) {$8_{20}$};
                \node[inner sep=0pt] at (5.5,-2.5) {$10_{87}$};
                        \node[inner sep=0pt] at (11,-2.5) {$10_{140}$};

     \node[inner sep=0pt] at (0,-5)
    {\includegraphics[width=.27\textwidth]{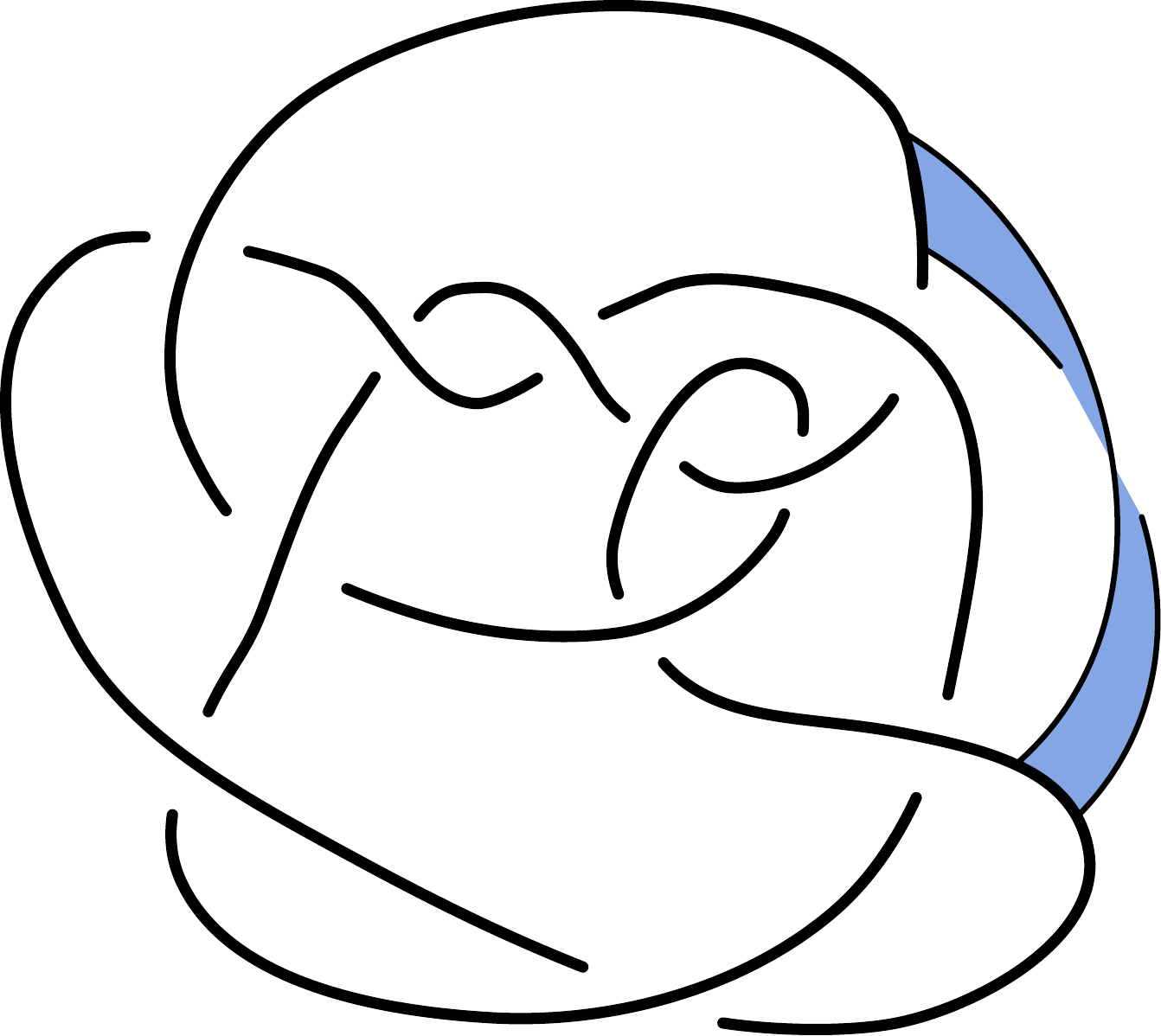}};
    \node[inner sep=0pt] at (5.5,-5)
    {\includegraphics[width=.27\textwidth]{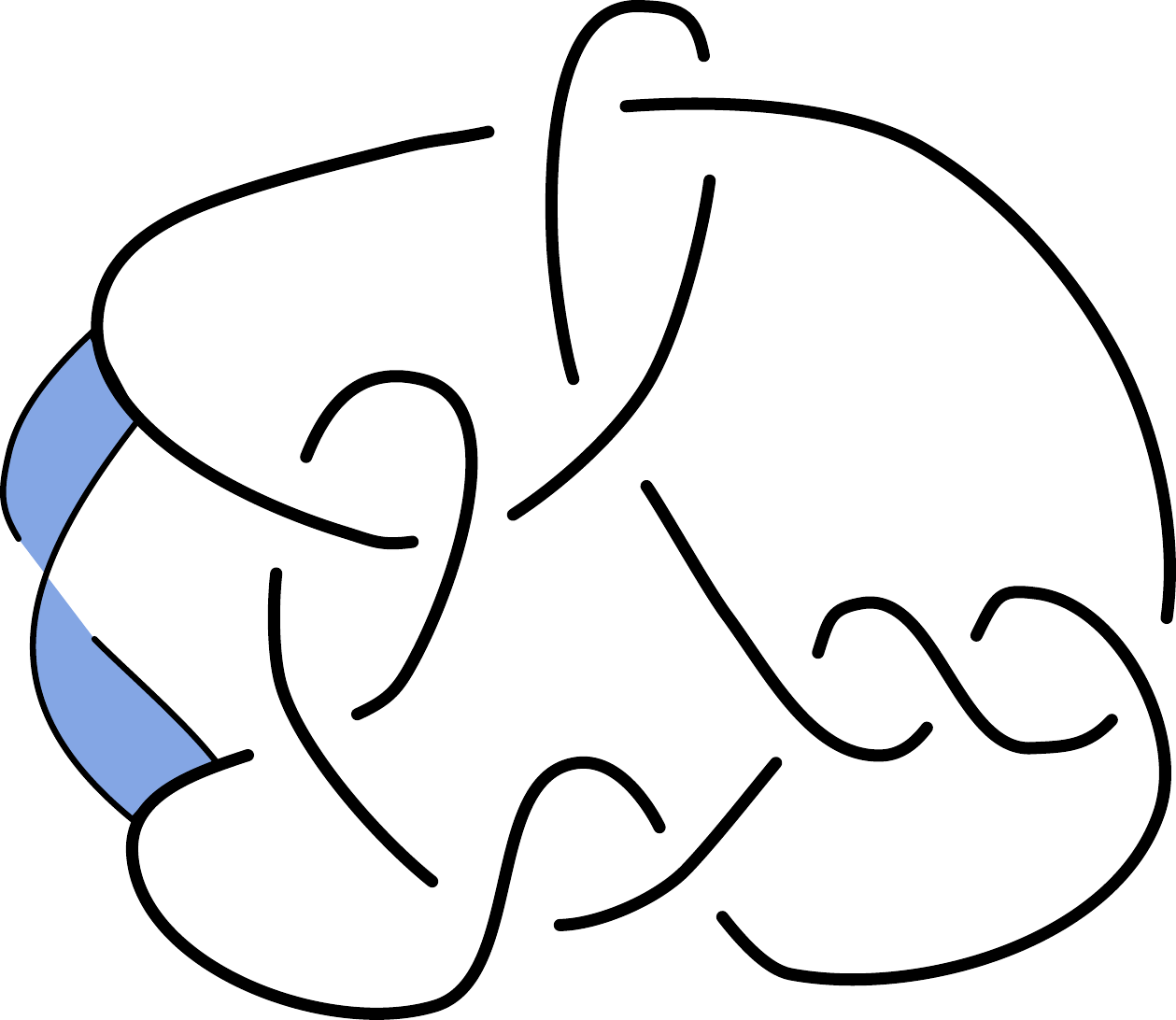}};
    \node[inner sep=0pt] at (11,-5)
    {\includegraphics[width=.27\textwidth]{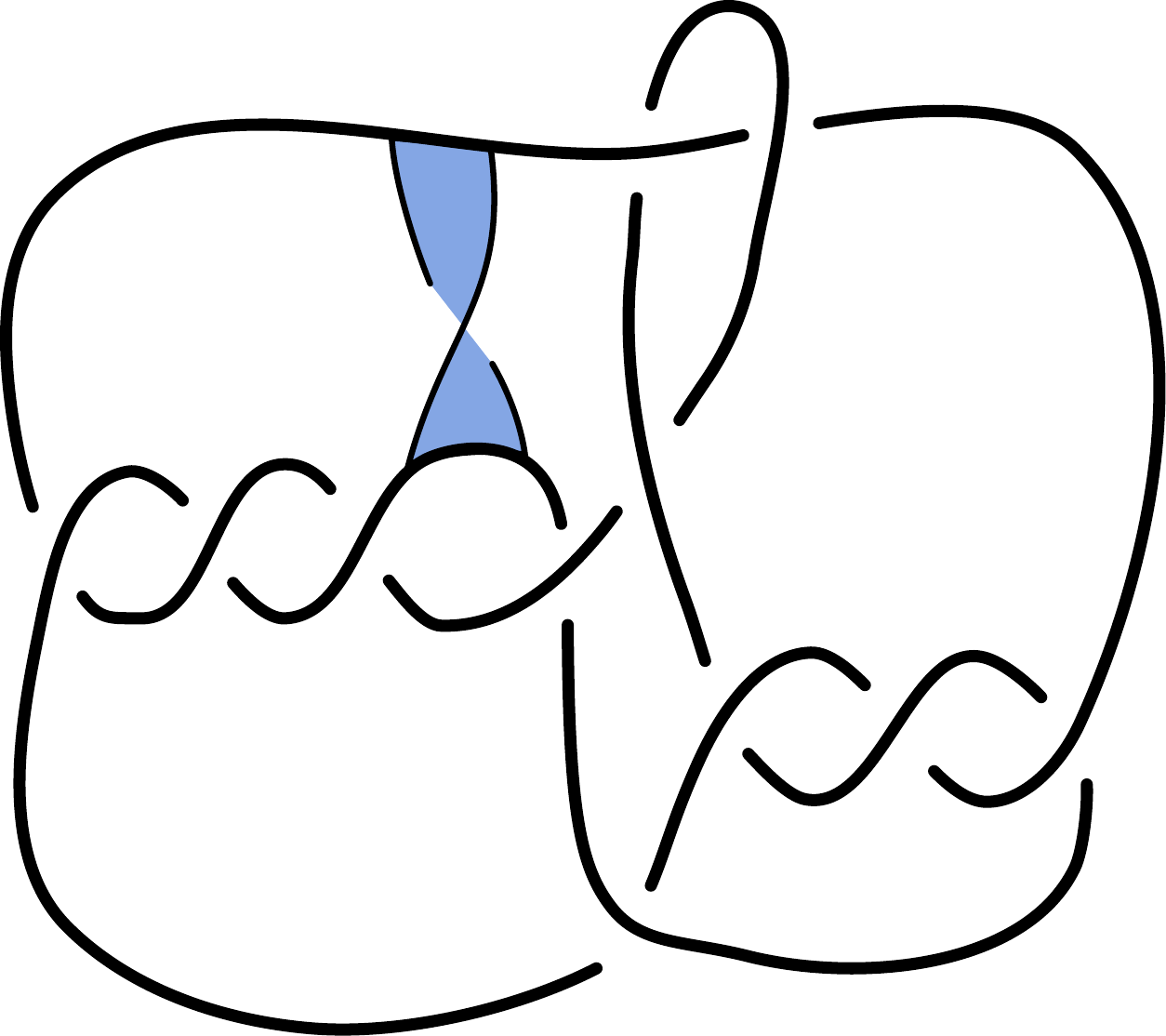}};

          \node[inner sep=0pt] at (0,-7.3) {$11a_{28}$};
                \node[inner sep=0pt] at (5.5,-7.3) {$11a_{58}$};
                        \node[inner sep=0pt] at (11,-7.3) {$11a_{165}$};

         \node[inner sep=0pt] at (0,-10.5)
    {\includegraphics[width=.31\textwidth]{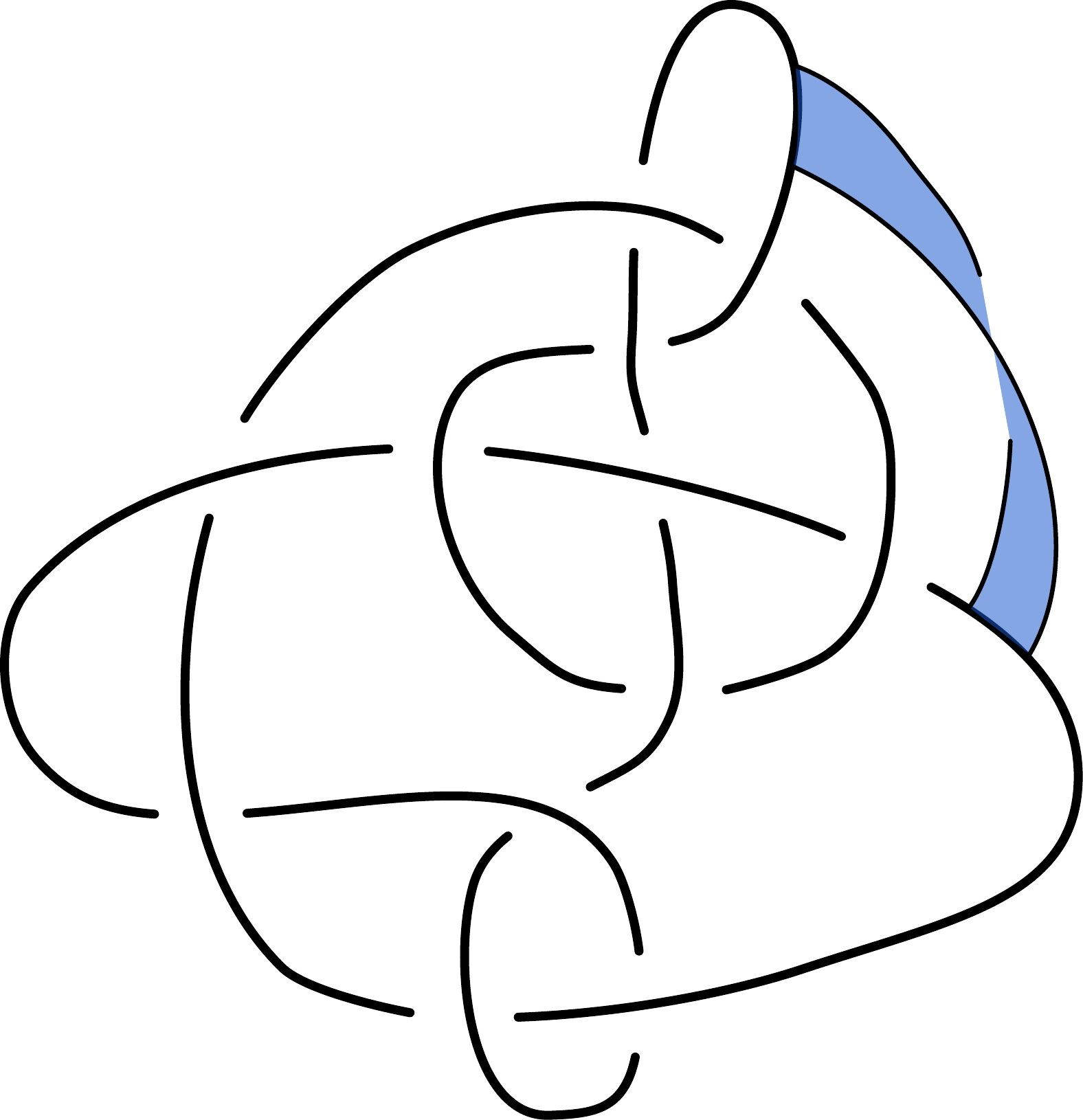}};
    \node[inner sep=0pt] at (5.5,-10.5)
    {\includegraphics[width=.24\textwidth]{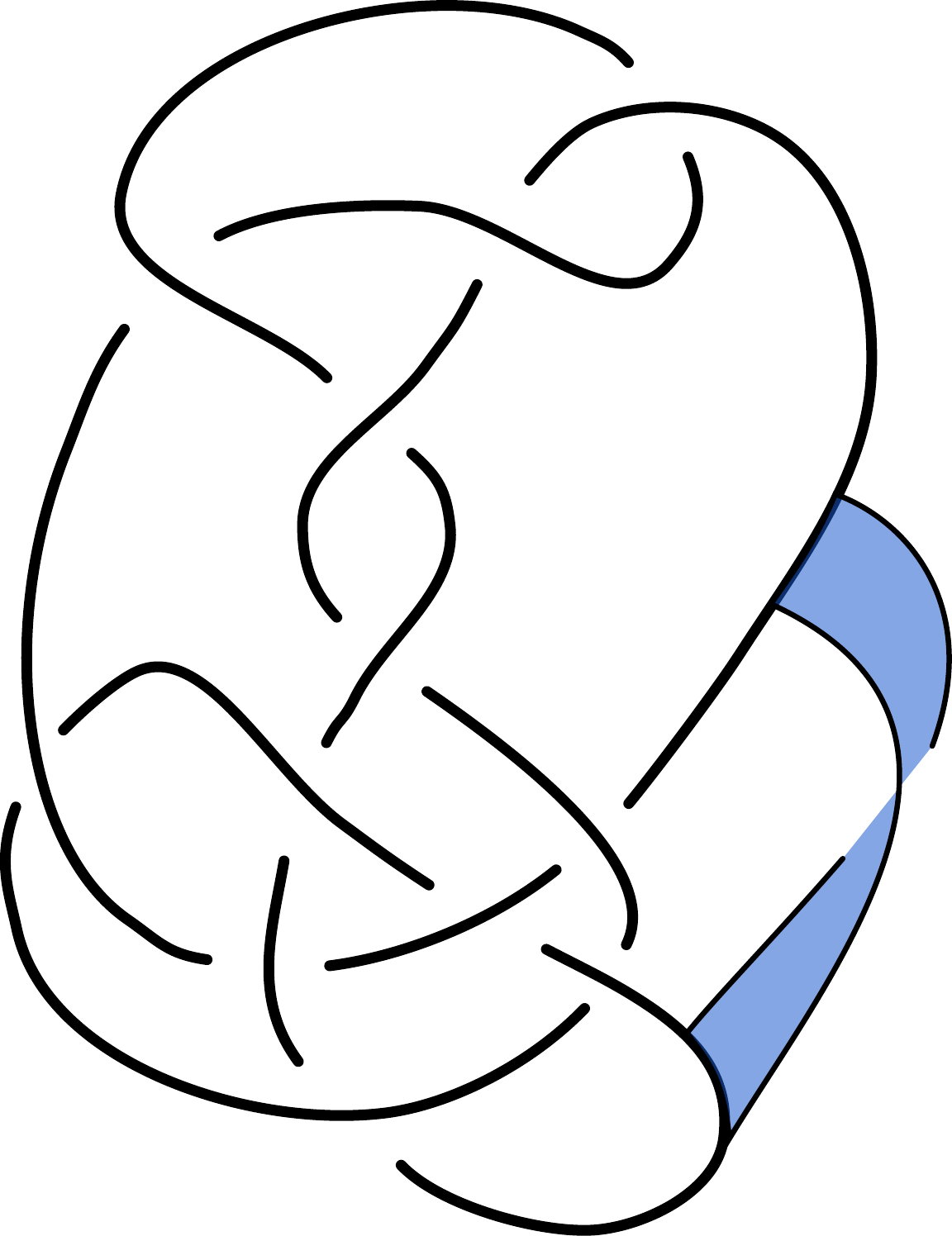}};
    \node[inner sep=0pt] at (11,-10.5)
    {\includegraphics[width=.29\textwidth]{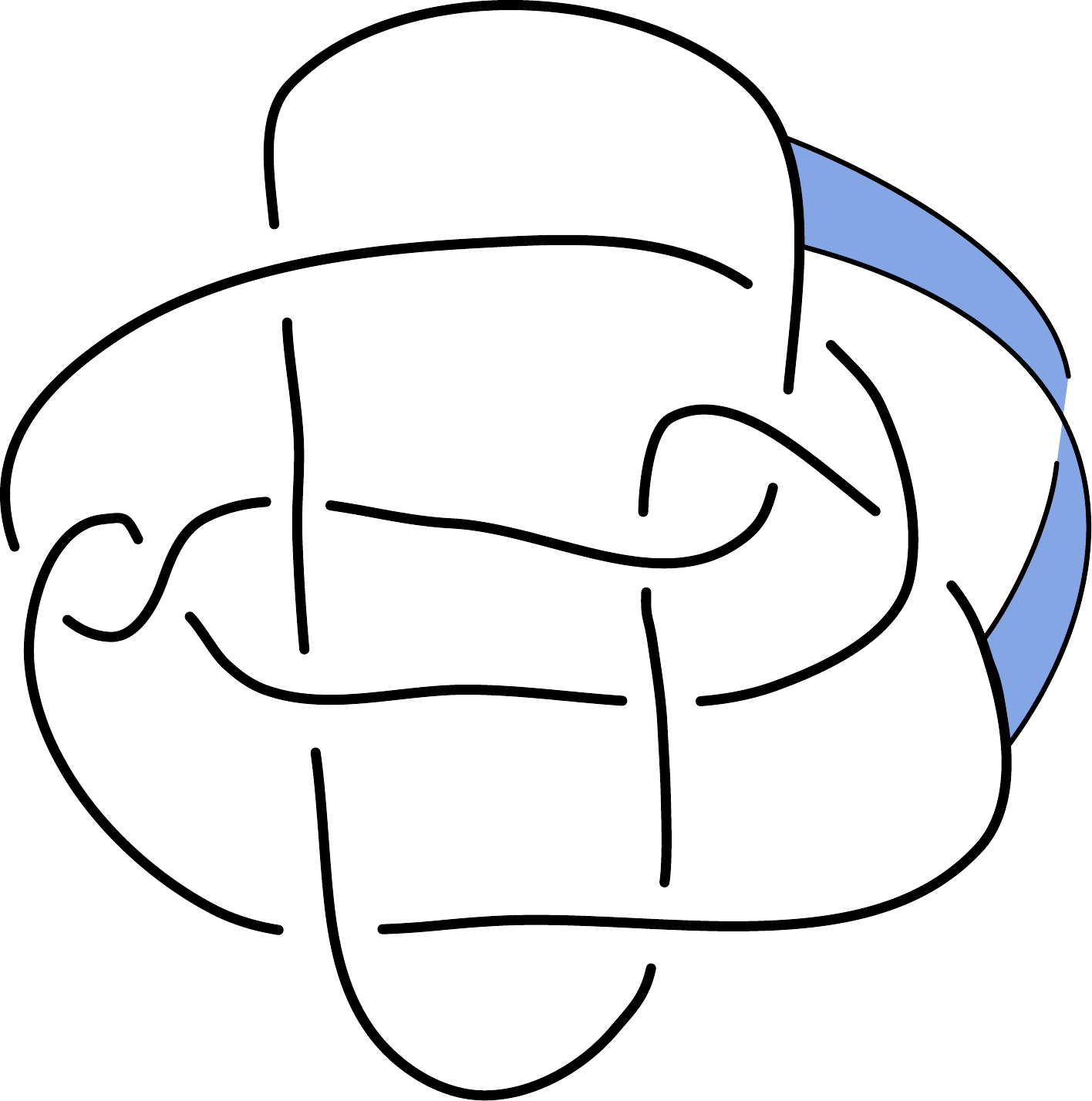}};

              \node[inner sep=0pt] at (0,-13.4) {$12a_{189}$};
                \node[inner sep=0pt] at (5.5,-13.4) {$12a_{377}$};
                        \node[inner sep=0pt] at (11,-13.4) {$12a_{979}$};

             \node[inner sep=0pt] at (0,-16)
    {\includegraphics[width=.26\textwidth]{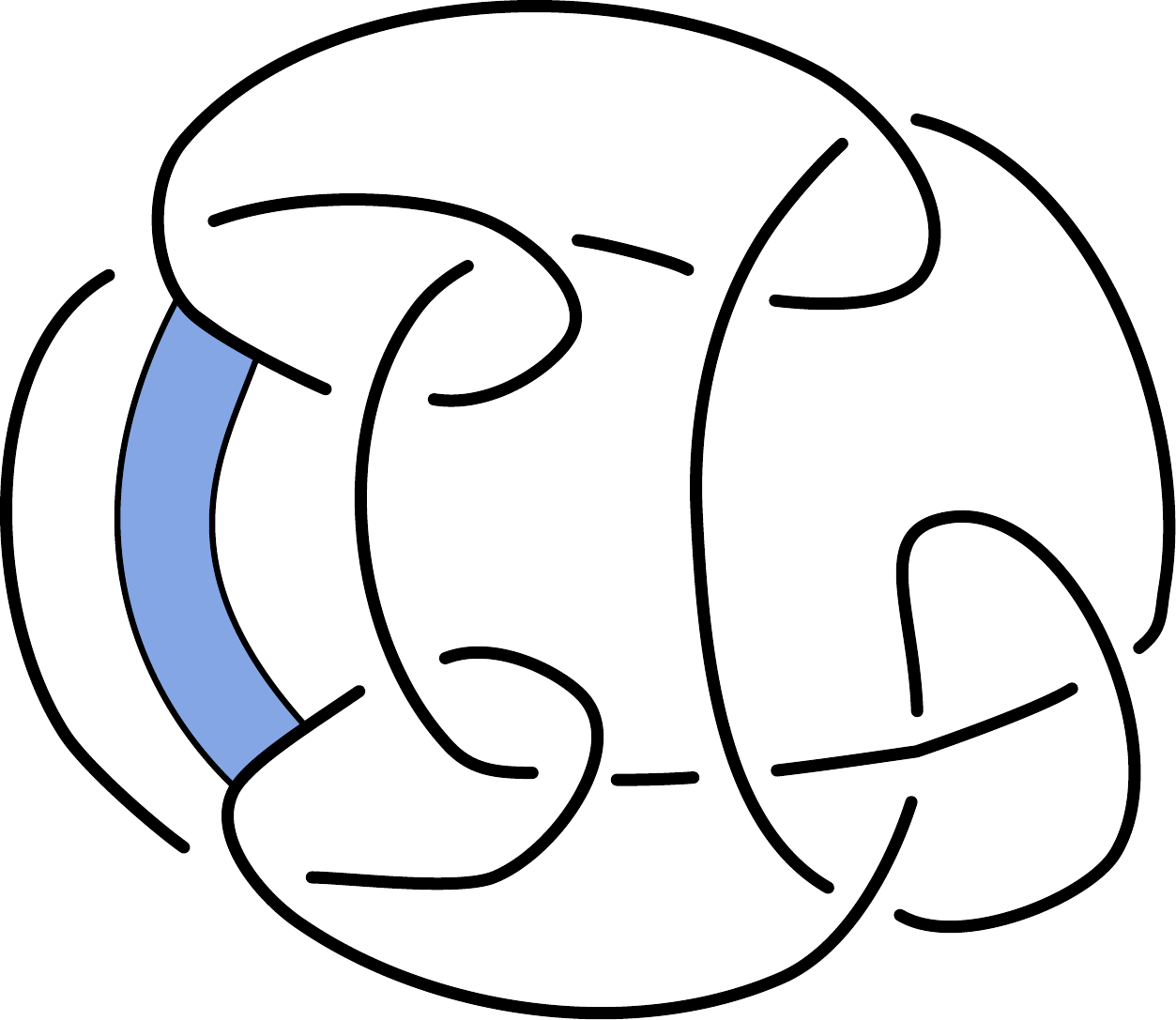}};
    \node[inner sep=0pt] at (5.5,-16)
    {\includegraphics[width=.26\textwidth]{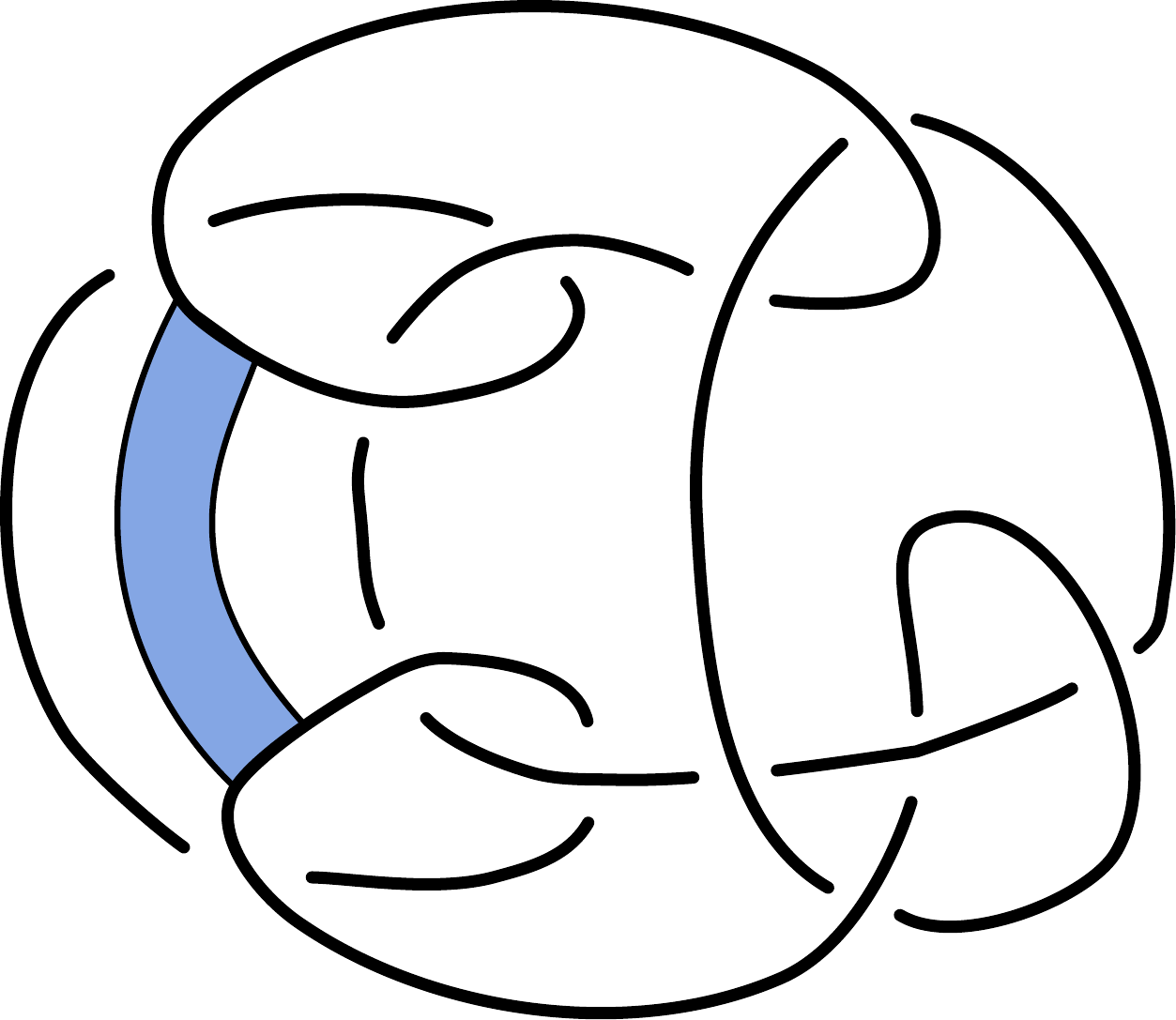}};
    \node[inner sep=0pt] at (11,-16)
    {\includegraphics[width=.30\textwidth]{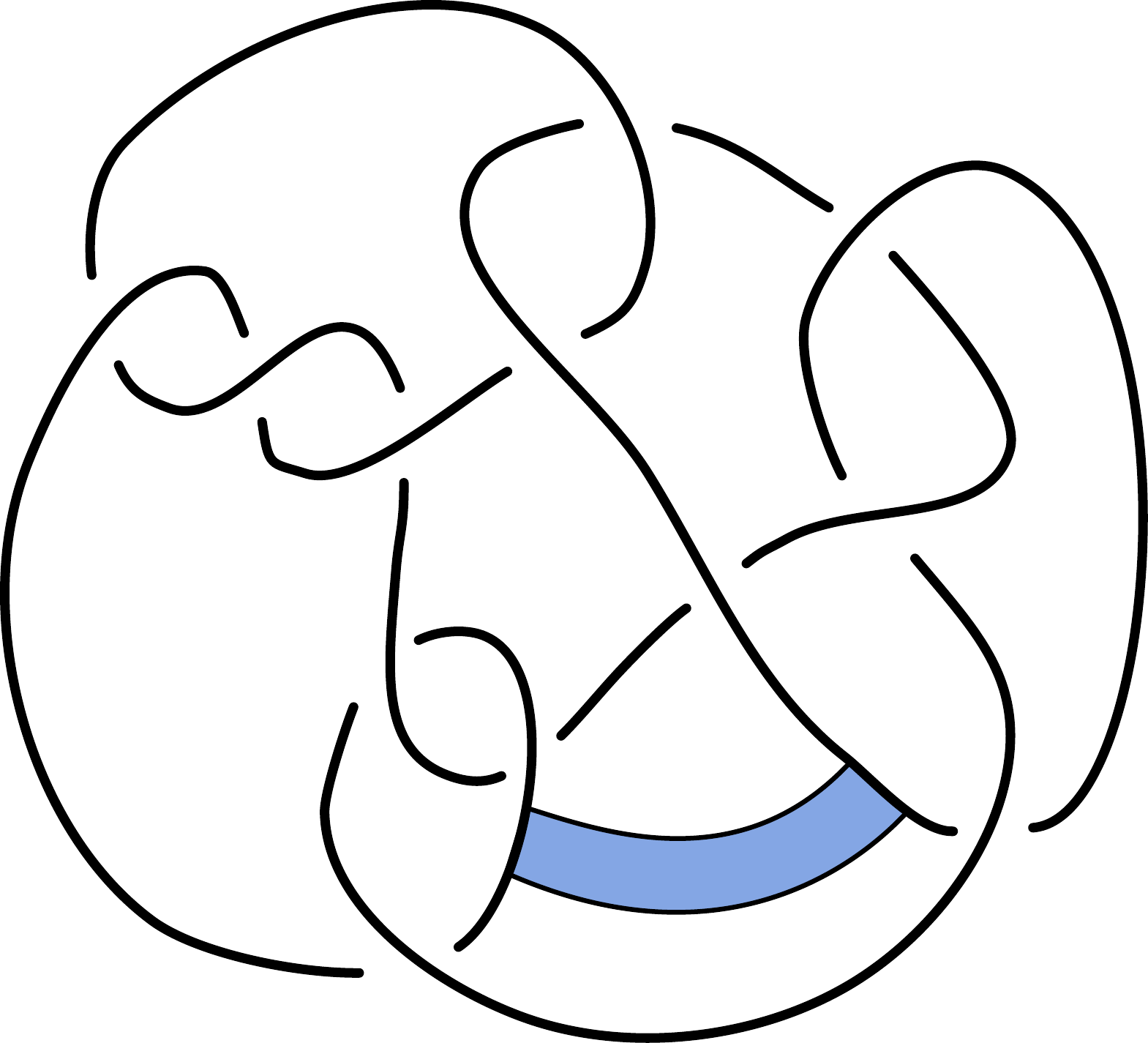}};

                  \node[inner sep=0pt] at (0,-18.4) {$12n_{56}$};
                \node[inner sep=0pt] at (5.5,-18.4) {$12n_{57}$};
                        \node[inner sep=0pt] at (11,-18.4) {$12n_{62}$};

    \end{tikzpicture}
    \end{center}
  \caption{\label{fig:1} Band moves for the proof of \fullref{prop:bandmoves}.}
  \end{figure}
    \begin{figure}
  \begin{tikzpicture}

      \node[inner sep=0pt] at (0,0)
    {\includegraphics[width=.31\textwidth]{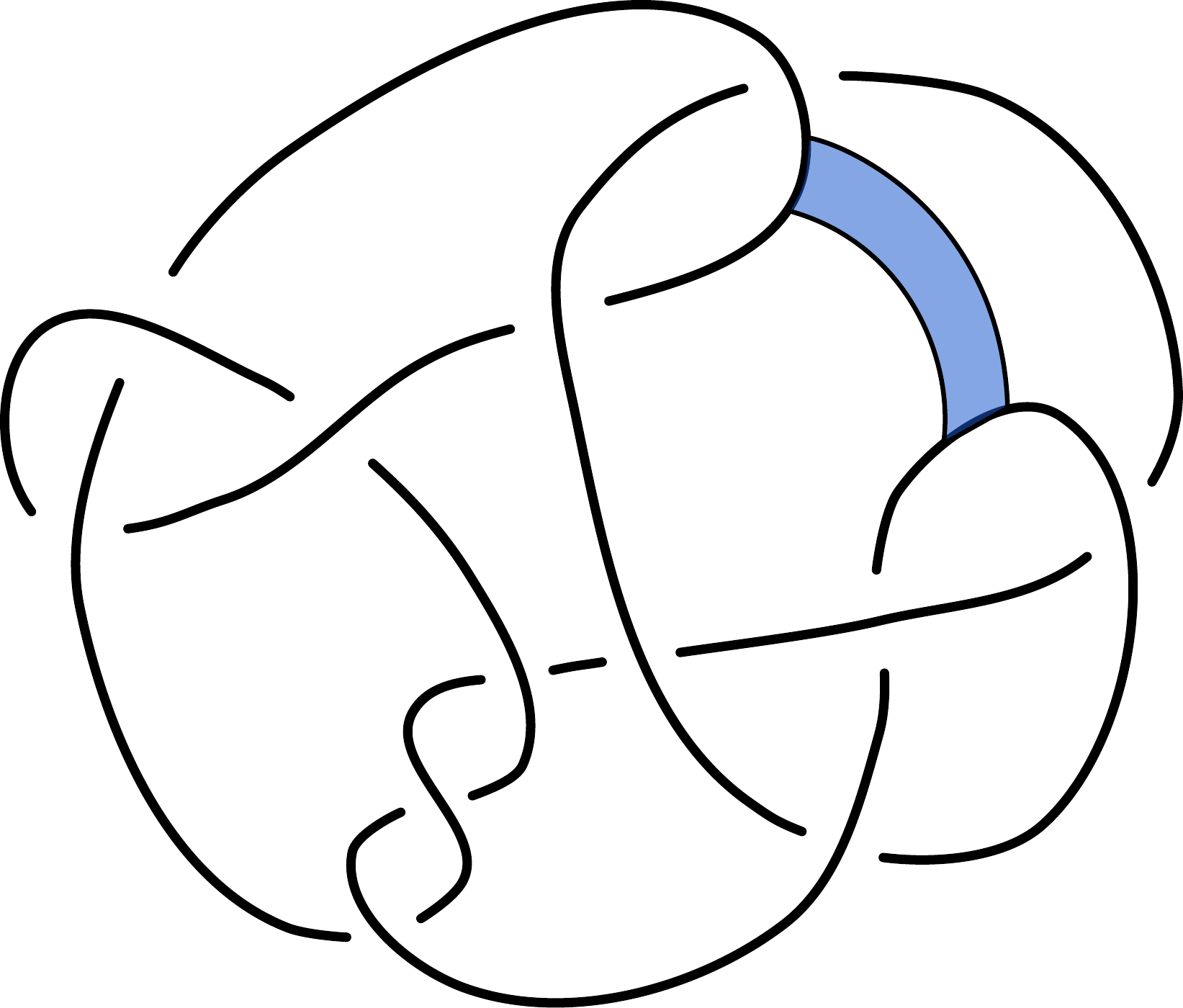}};
    \node[inner sep=0pt] at (5.5,0)
    {\includegraphics[width=.27\textwidth]{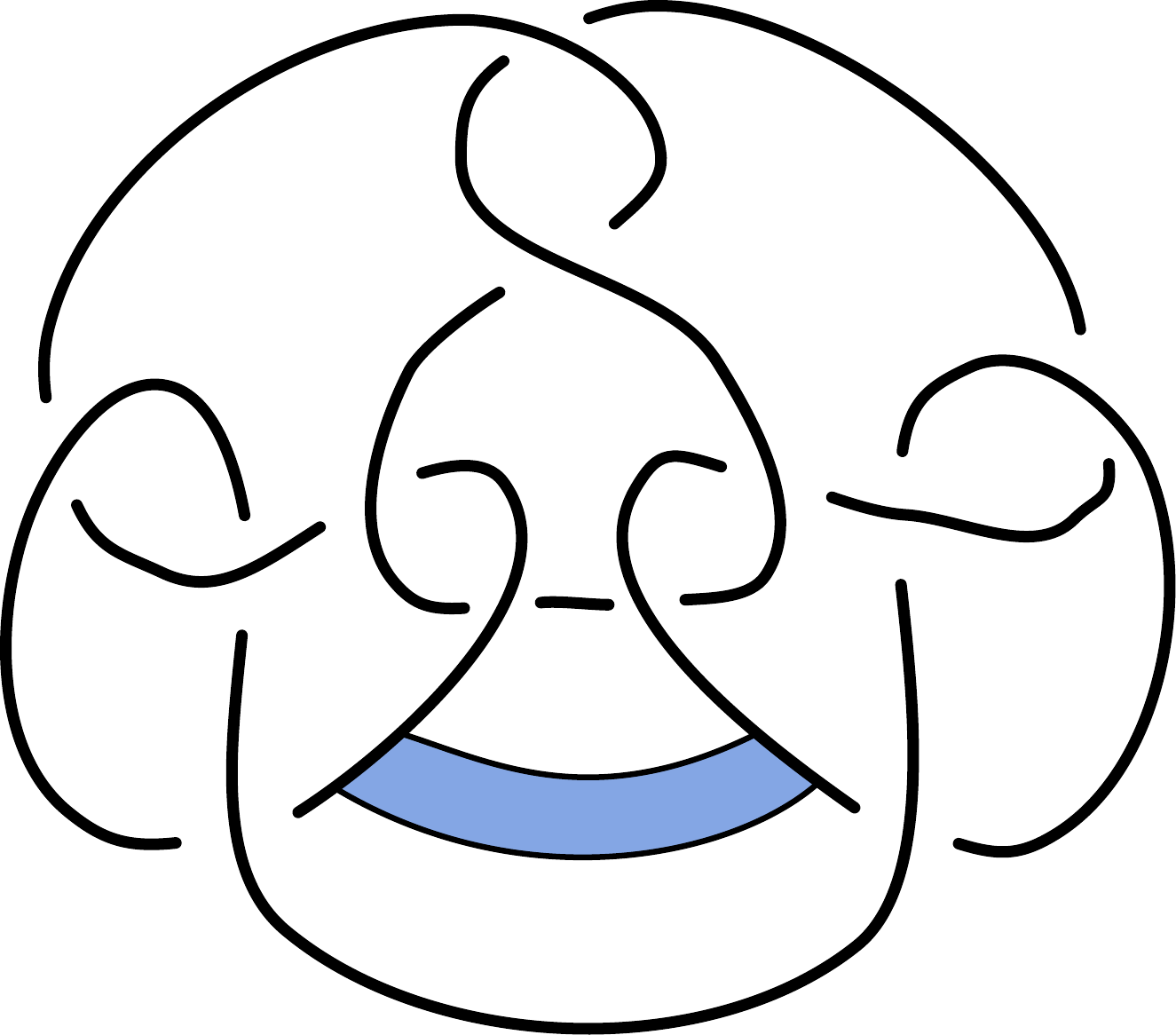}};
    \node[inner sep=0pt] at (11,0)
    {\includegraphics[width=.25\textwidth]{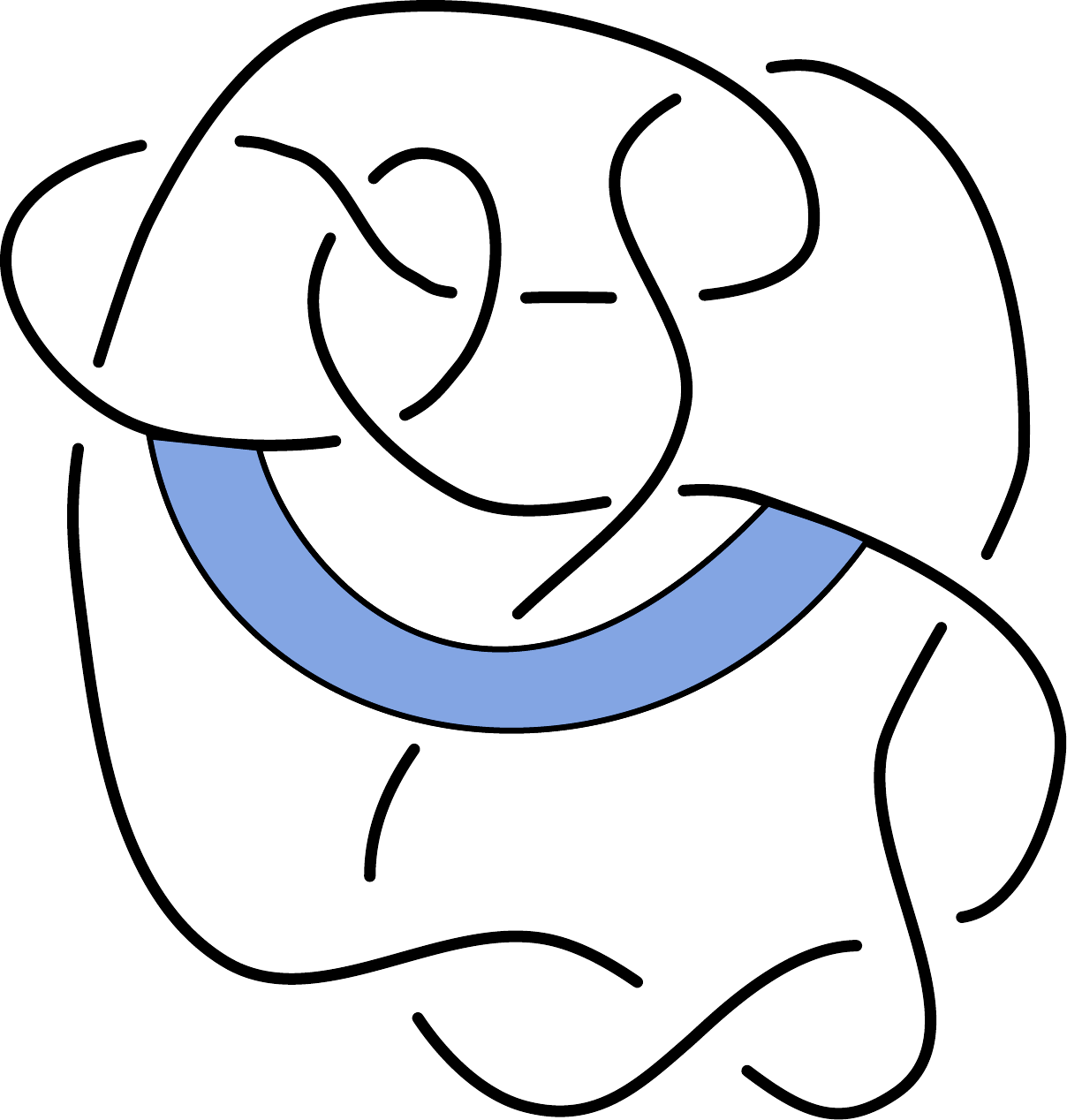}};

            \node[inner sep=0pt] at (0,-2.5) {$12n_{66}$};
                \node[inner sep=0pt] at (5.5,-2.5) {$12n_{87}$};
                        \node[inner sep=0pt] at (11,-2.5) {$12n_{106}$};

     \node[inner sep=0pt] at (0,-5.5)
    {\includegraphics[width=.32\textwidth]{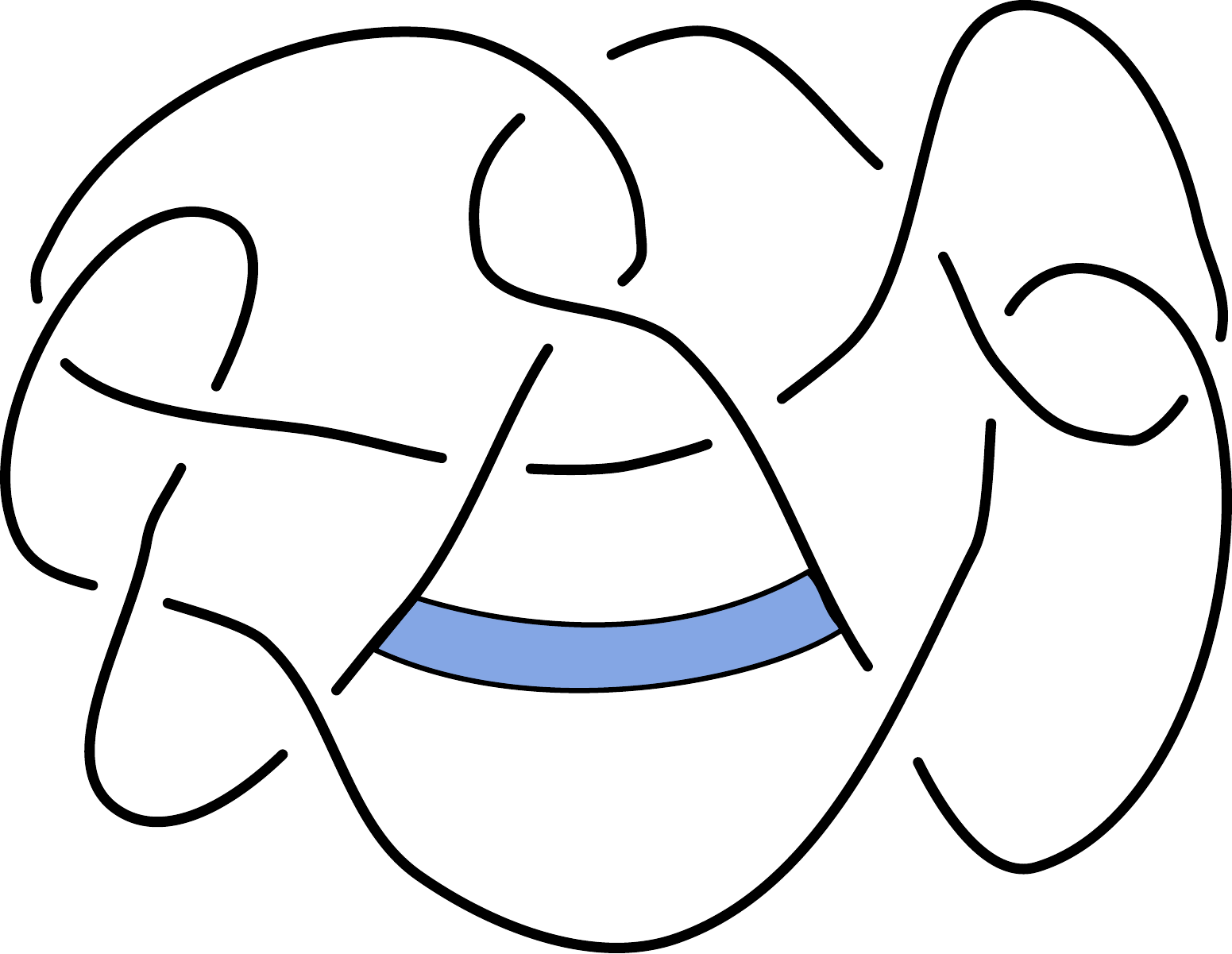}};
    \node[inner sep=0pt] at (5.5,-5.5)
    {\includegraphics[width=.25\textwidth]{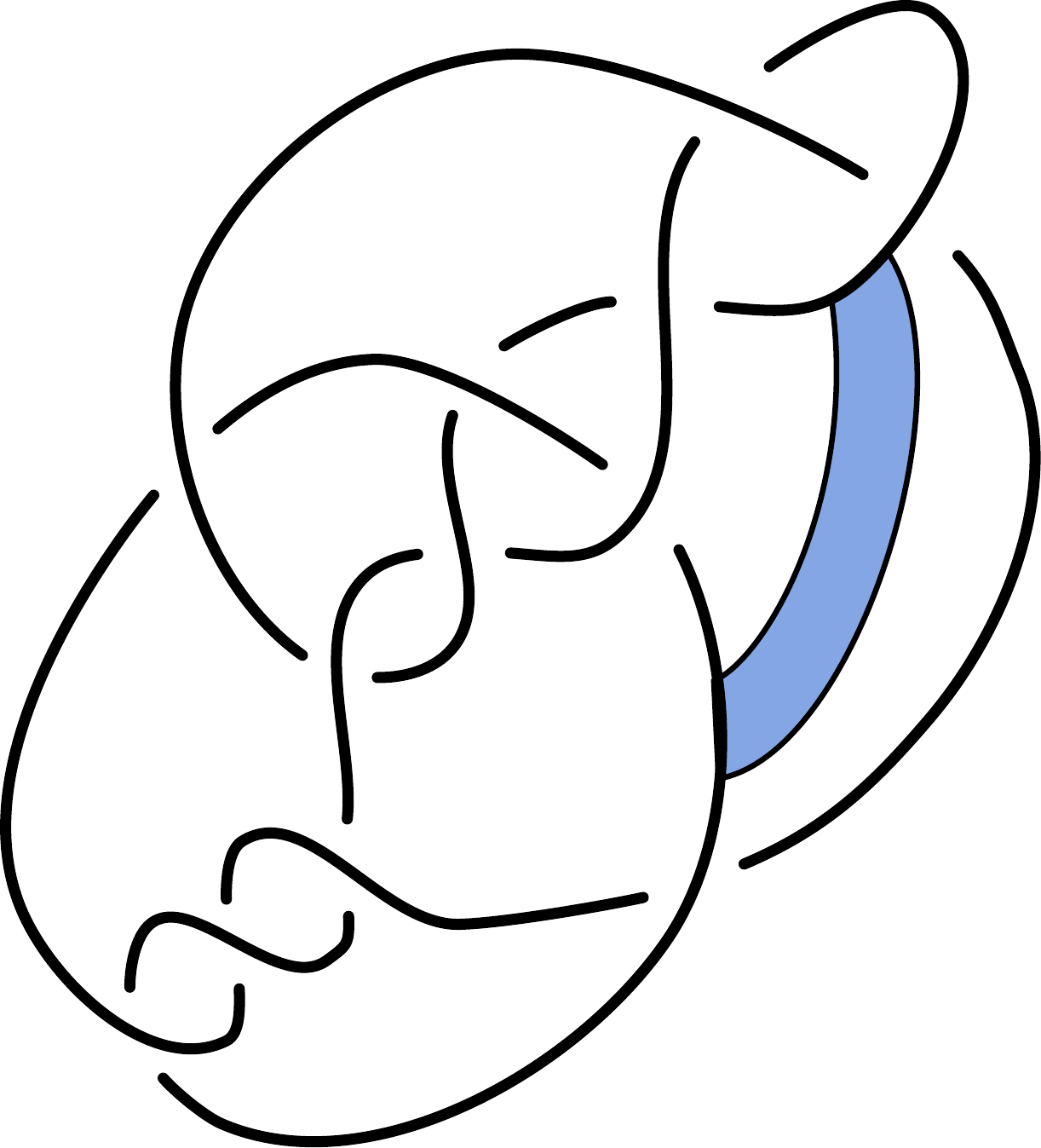}};
    \node[inner sep=0pt] at (11,-5.5)
    {\includegraphics[width=.31\textwidth]{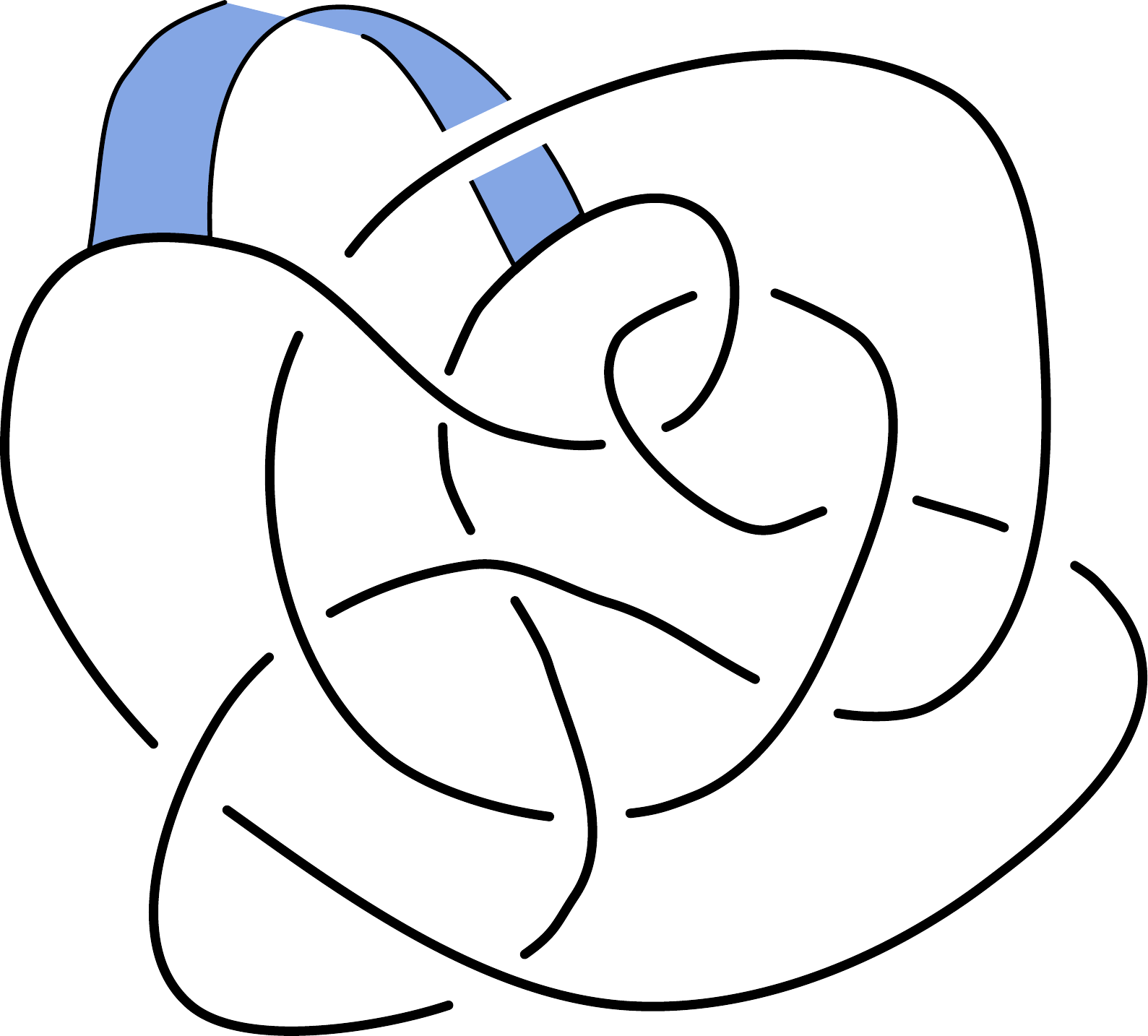}};

              \node[inner sep=0pt] at (0,-8) {$12n_{288}$};
                \node[inner sep=0pt] at (5.5,-8) {$12n_{501}$};
                        \node[inner sep=0pt] at (11,-8) {$12n_{504}$};

         \node[inner sep=0pt] at (0,-11)
    {\includegraphics[width=.3\textwidth]{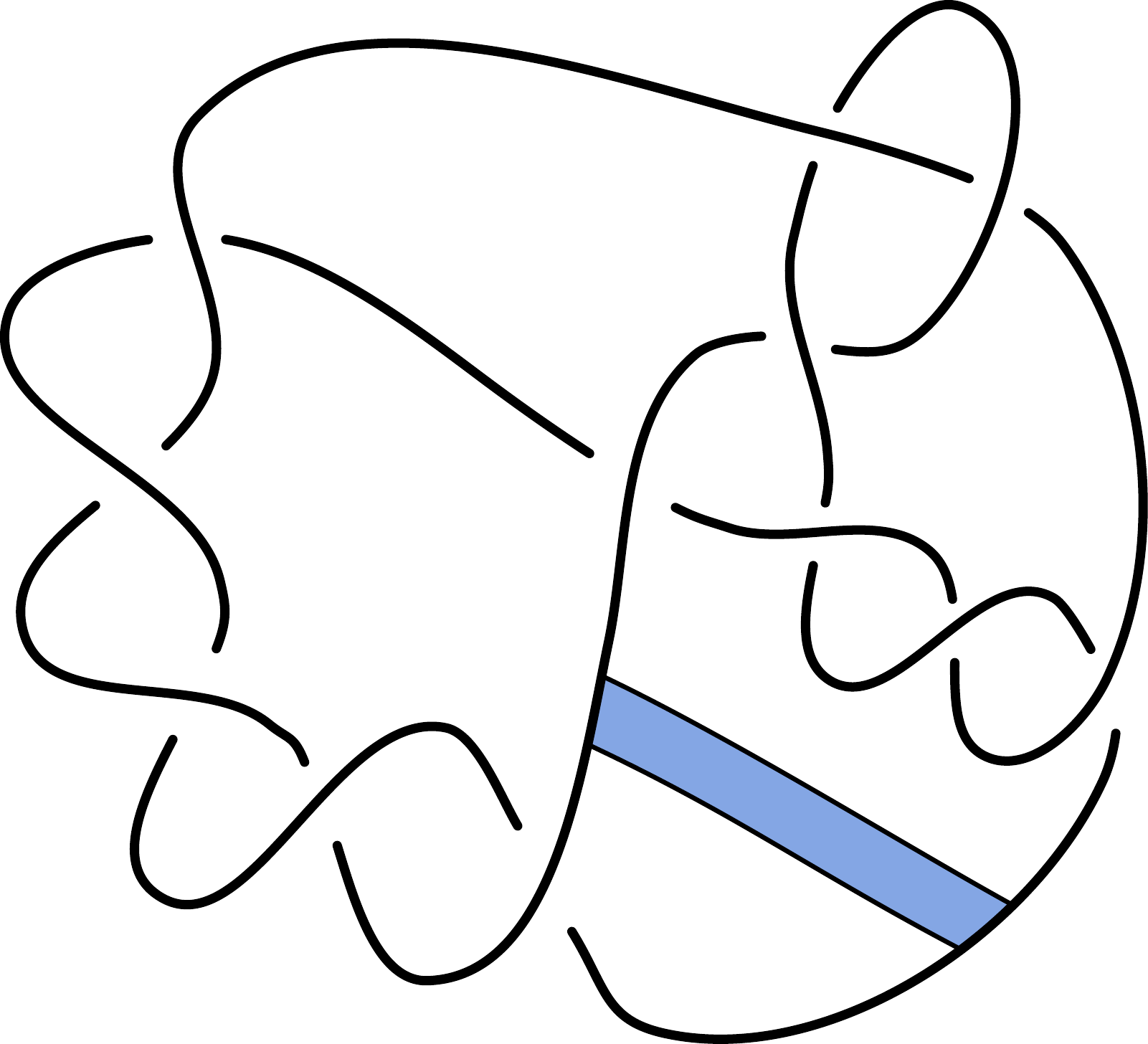}};
    \node[inner sep=0pt] at (5.5,-11)
    {\includegraphics[width=.3\textwidth]{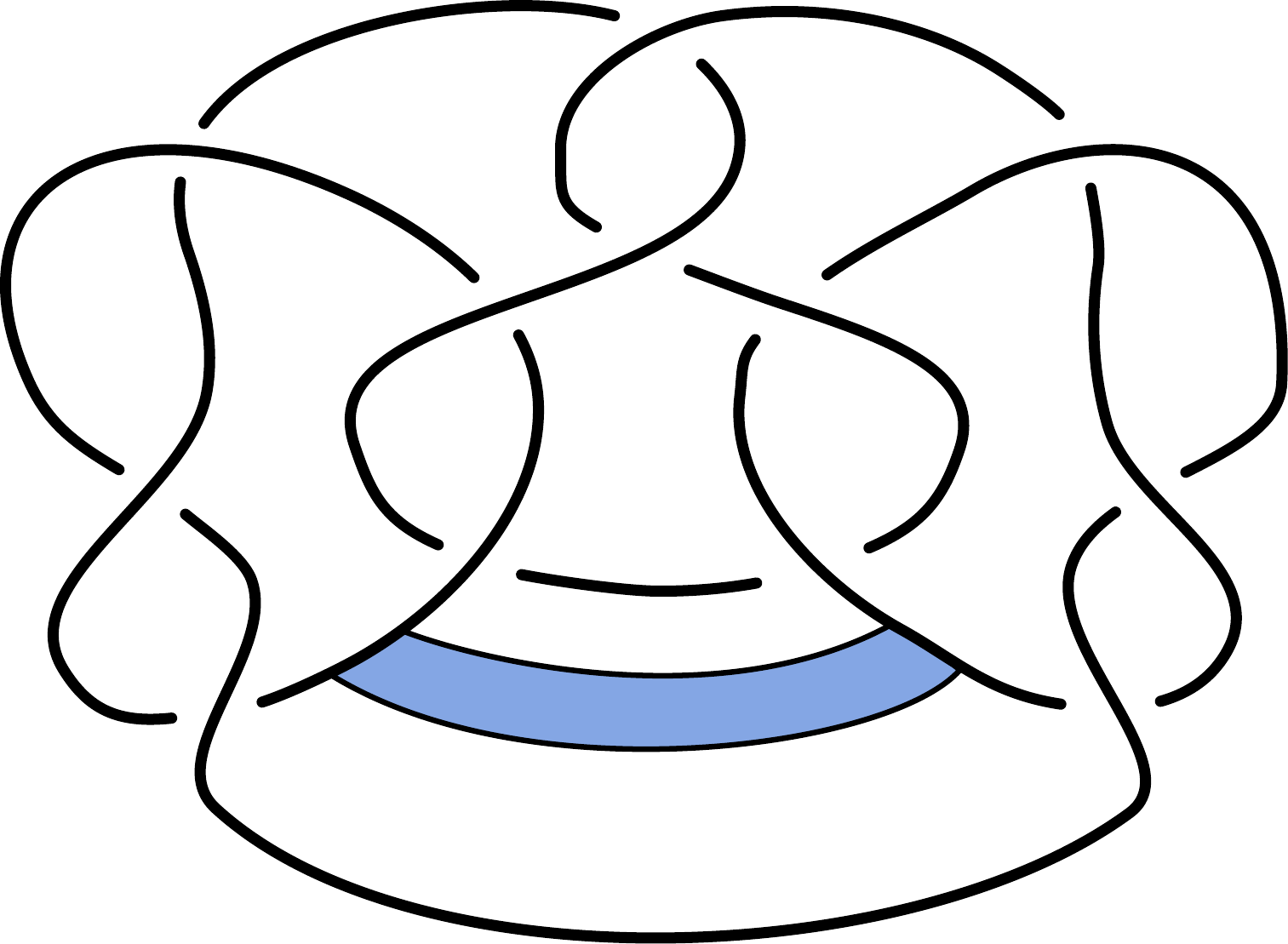}};
    \node[inner sep=0pt] at (11,-11)
    {\includegraphics[width=.29\textwidth]{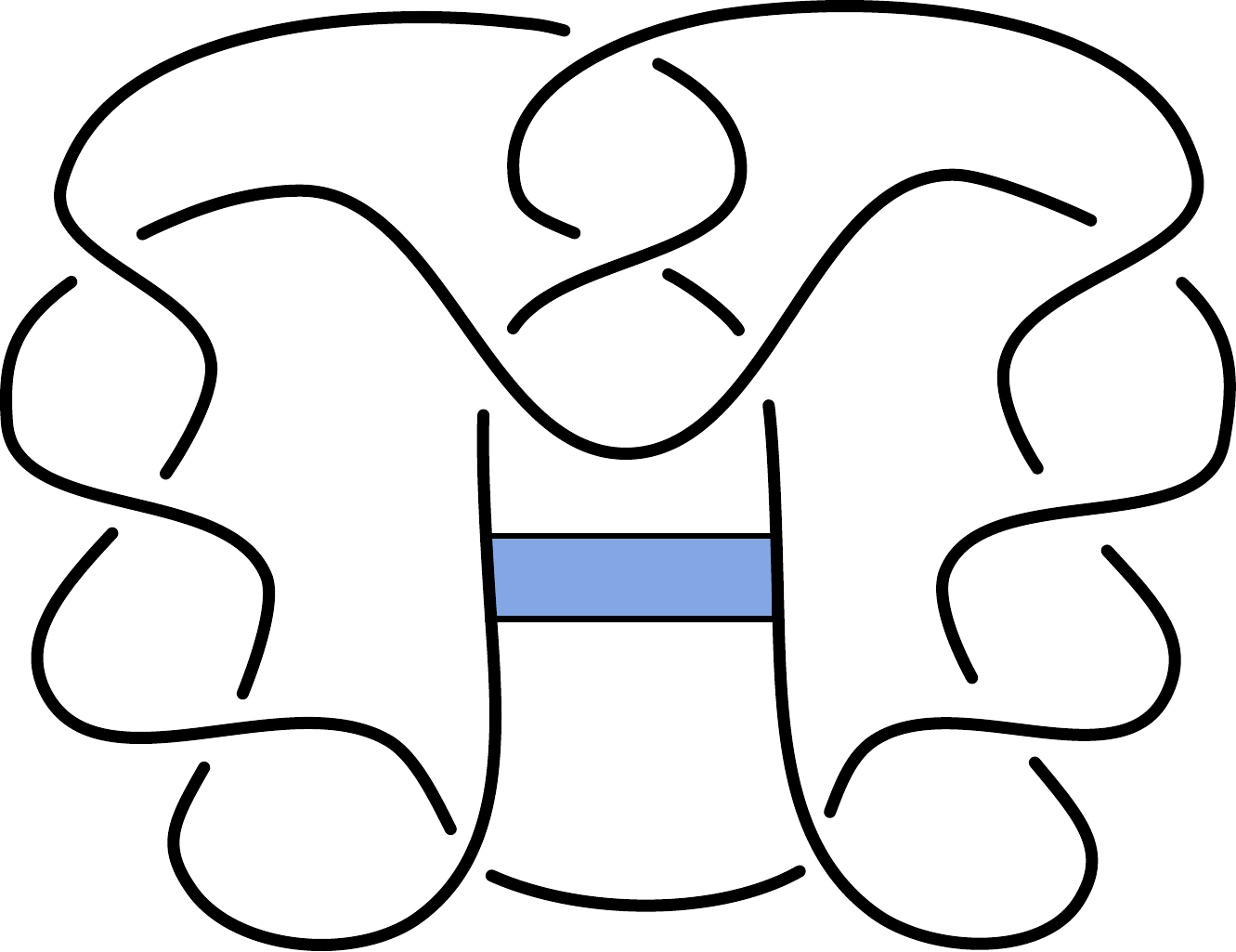}};

                  \node[inner sep=0pt] at (0,-13.4) {$12n_{582}$};
                \node[inner sep=0pt] at (5.5,-13.4) {$12n_{670}$};
                        \node[inner sep=0pt] at (11,-13.4) {$12n_{721}$};

    \end{tikzpicture}
  \caption{\label{fig:2} More band moves for the proof of \fullref{prop:bandmoves}.}
  \end{figure}
  \end{proof}

\bibliographystyle{gtart}
\bibliography{doubly-slice-genus}

\end{document}